\documentclass[12pt]{amsart}
\usepackage{amsmath,amsthm,amsfonts,amssymb}
\usepackage[
margin=2.5cm]{geometry}
\usepackage[usenames,dvipsnames]{color}
\usepackage{mathrsfs}
\usepackage[colorlinks,allcolors=blue]{hyperref}
\usepackage{enumitem}
\setlist[enumerate]{%
  leftmargin=*,%
  label=(\roman*)} 
\setlist[itemize]{leftmargin=*}
\usepackage[mathscr]{eucal}
\usepackage{multirow, array}

\usepackage{tikz,xcolor}
\usetikzlibrary{positioning, backgrounds}

\usepackage{
  tikz-qtree}
\usetikzlibrary{matrix,arrows,trees,shapes}
\usepackage{xcolor}

\usepackage{fancyvrb} 
\RecustomVerbatimEnvironment{Verbatim}{Verbatim}%
{xleftmargin=0pt,baselinestretch=0.85}

\usepackage{microtype}

\usepackage{bm} 
\usepackage{bbm} 

\usepackage[all]{xy} 
\usepackage[T1]{fontenc} 

\usepackage{graphics} 

\usepackage[yyyymmdd]{datetime}

\usepackage{fancyhdr}
\newtheorem{theorem}{Theorem}[section]
\newtheorem{lemma}[theorem]{Lemma}
\newtheorem{proposition}[theorem]{Proposition}
\newtheorem{corollary}[theorem]{Corollary}
\newtheorem{question}[theorem]{Question}

\theoremstyle{definition}%
\newtheorem{example}[theorem]{Example}
\newtheorem{remark}[theorem]{Remark}

\DeclareMathOperator{\hilb}{Hilb}
\DeclareMathOperator{\rank}{rank}
\DeclareMathOperator{\Kos}{Kos}

\DeclareMathOperator{\GL}{GL}%
\DeclareMathOperator{\Gr}{Gr}%
\DeclareMathOperator{\Hom}{Hom}%
\DeclareMathOperator{\Soc}{Soc}%
\DeclareMathOperator{\Ann}{Ann}%
\DeclareMathOperator{\id}{id}%

\newcommand{\fmm}{\mathfrak{m}}
\newcommand{\cP}{\mathcal{P}}%
\newcommand{\kk}{\mathbbm{k}}
\newcommand{\NN}{\mathbb{N}}
\newcommand{\QQ}{\mathbb{Q}}
\newcommand{\ZZ}{\mathbb{Z}}
\newcommand{\bbA}{\mathbb{A}}

\DeclareMathOperator{\sgn}{sign}

\newcommand{\llrr}[1]{ \langle #1 \rangle } 

\newif\ifhascomments \hascommentstrue
\ifhascomments
\else
\fi

\begin{document}

\begin{abstract}
  Bhargava and the first-named author of this paper introduced a
  functorial Galois closure operation for finite-rank ring extensions,
  generalizing constructions of Grothendieck and Katz--Mazur. In this
  paper, we generalize Galois closures and apply them to construct a new infinite family
  of irreducible components of Hilbert schemes of points. We show that
  these components are elementary, in the sense that they parametrize
  algebras supported at a point. Furthermore, we produce secondary families of elementary components obtained from Galois closures by modding out by suitable socle elements.
\end{abstract}

\title{Galois closures and elementary components of Hilbert schemes of points}

\author{Matthew~Satriano}
\thanks{MS was partially supported by a Discovery Grant from the
  National Science and Engineering Research Council of Canada and a Mathematics Faculty Research Chair.}
\address{Matthew Satriano, Department of Pure Mathematics, University
  of Waterloo}
\email{msatrian@uwaterloo.ca}

\author{Andrew~P.~Staal}
\address{Andrew~P.~Staal, Department of Pure Mathematics, University of
  Waterloo}
\email{andrew.staal@uwaterloo.ca}
\maketitle

\section{Introduction}
\label{sec:intro}


First introduced by Grothendieck \cite{Grothendieck--1961} in 1961,
Hilbert schemes have played a central role in algebraic geometry,
commutative algebra, and algebraic combinatorics. They are fundamental
building blocks in the construction of many moduli spaces, and have
seen numerous broad-ranging applications from the McKay correspondence
\cite{Ito-Nakajima--2000,Bridgeland-King-Reid--2001} to Haiman's proof
of the Macdonald positivity conjecture \cite{Haiman--2001}. While
Hilbert schemes of points on smooth surfaces are irreducible
\cite{Fogarty--1968}, in contrast, Iarrobino \cite{Iarrobino--1972,
  Iarrobino--1973} and Iarrobino--Emsalem
\cite{Iarrobino--Emsalem--1978} showed that for $n\geq3$ and $d$
sufficiently large, the Hilbert scheme of points $\hilb^d(\bbA^n)$ is
reducible. It has remained a notoriously difficult problem to describe
the structure of the irreducible components of $\hilb^d(\bbA^n)$.

The study of all irreducible components of $\hilb^d(\bbA^n)$ can be
reduced to that of \emph{\bfseries elementary components}, namely
irreducible components parametrizing subschemes supported at a point;
this is due to the fact that, generically, every component is
\'etale-locally the product of elementary ones. To date, relatively
few constructions of irreducible components exist in the literature
\cite{Iarrobino--1984, Shafarevich--1990, Iarrobino--Kanev--1999,
  CEVV--2009, Erman--Velasco--2010, Huibregtse--2017,
  Jelisiejew--2019, Jelisiejew--2020, Huibregtse--2021,
  Satriano--Staal--2023} and even less is known about elementary
components. In fact, it was an open question for over 30 years to
determine whether there exists an irreducible component of the
punctual Hilbert scheme $\hilb^d(\mathcal{O}_{\bbA^n,0})$ of dimension
less than $(n-1)(d-1)$, see \cite[p.\ 310]{Iarrobino--1987},
cf.~\cite[p.\ 186]{Iarrobino--Emsalem--1978}; this question was only
recently answered by the authors of this paper \cite[Theorem
  1.2]{Satriano--Staal--2023}.


\vspace{0.6em}

In the current paper, we obtain a new infinite family of elementary
components by applying a seemingly unrelated construction introduced
by Bhargava and the first-named author in \cite{bhargava-satriano}:~a
functorial \emph{\bfseries Galois closure} operation for ring
extensions that commutes with arbitrary base change. This Galois
closure operation generalizes constructions appearing in work of
Grothendieck \cite[Chapter IV, Lemma 1]{Chevalley--1958} and
Katz--Mazur \cite[\S1.8.2]{Katz--Mazur--1985}; variants of the Galois
closure were used in Bhargava's groundbreaking work
\cite{BhargavaIII--2004,BhargavaIV--2008} and recently non-commutative
generalizations of the Galois closure were applied by Ho and the
first-named author \cite{Ho--Satriano-2020} to uniformly construct
many of the representations arising in arithmetic invariant theory,
including many Vinberg representations.

Before stating our main results, let us elucidate the connection
between Galois closures and Hilbert schemes of points. First, given
any morphism of rings $B\to A$ that realizes $A$ as a free $B$-module
of rank $n$, the Galois closure\footnote{also referred to as the
$S_n$-closure} $G(A/B)$ is a $B$-algebra which comes equipped with an
action of the symmetric group $S_n$. 
Furthermore, if $A$ is a quotient
of $B[x_1,\dots,x_r]$,
then $G(A/B)$ is naturally a quotient of $B[x_{i,j}]$, where $1\leq
i\leq r$ and $1\leq j\leq n$.  The ideal defining $G(A/B)$ contains
the linear form $\sum_j x_{i,j}$ for every $i$, thus, upon eliminating
the variables $x_{i,n}$, we see
\[
[A]\in\hilb^n(\bbA^r_\kk)\ \ \Longrightarrow\ \ [G(A/\kk)]\in\hilb^{pts}(\bbA^{r(n-1)}_\kk)
\]
where $\kk$ is a field. For example, $G(\kk^n/\kk)\simeq \kk^{n!}$ as a $\kk$-algebra, and more specifically, is given by the regular representation. In particular, the Galois closure operation takes configurations of $n$ distinct points in $\bbA^r_\kk$ to configurations of $n!$ distinct points in $\bbA^{r(n-1)}_\kk$.

Of notable interest for us is that $G(A/B)$ need not be a flat
$B$-module even though $A$ is flat. While this may seem unsightly at
first, it turns out to be a useful feature in the study of Hilbert
schemes; indeed, this gives us a method to take any easily understood
algebra $A$ lying on the main component of the Hilbert scheme and
produce a new algebra $G(A/\kk)$ which has the potential to lie off of
the main component. This is particularly useful when one focuses
attention on those algebras $A$ defined by monomial ideals $I$; here
$G(A/\kk)$ carries a rich combinatorial structure coming both from the
$S_n$-action and the fact that $I$ is monomial.

The algebras that featured prominently in \cite{bhargava-satriano},
and those which play the most important role in our current paper, are
\[
A_{n}:=\kk[x_1,\dots,x_{n}]/(x_1,\dots,x_{n})^2.
\]
Within Poonen's moduli space of rank $n$ rings \cite{Poonen--2008},
every algebra degenerates to $A_{n-1}$, and hence $G(A_{n-1}/\kk)$ has the
largest possible dimension among Galois closures, see \cite[Theorem
  8]{bhargava-satriano}. One of the main theorems of
\cite{bhargava-satriano}, see Theorem 27 of (loc.~cit.), gives the decomposition of 
$G(A_{n-1}/\kk)$ into irreducible $S_n$-representations, and in particular, gives a combinatorial formula for $d(n):=\dim_\kk G(A_{n-1}/\kk)$.

It follows from the proof of \cite[Theorem 5]{bhargava-satriano} that
if $A$ is a flat $B$-algebra of rank $n\leq3$, then $G(A/B)$ is a flat
$B$-algebra of rank $n!$. In particular, if $A$ is a $\kk$-algebra of
rank at most $3$ and $A$ lies on the main component of the Hilbert
scheme, then $G(A/\kk)$ does as well. Thus, to find Galois closures
living off of the main component, one must consider rings of rank
$n\geq4$.

Throughout this paper, we work over an algebraically closed field
$\kk$ of characteristic $0$. Our main theorem is:

\begin{theorem}\label{thm:main}
  If $n\geq 4$, then every irreducible component of
  $\hilb^{d(n)}(\bbA^{(n-1)^2})$ containing $G(A_{n-1}/\kk)$ is
  elementary.
\end{theorem}

In fact, Theorem \ref{thm:main} follows from the more general
Theorem \ref{thm:main-higher-gc}, which we now describe.\footnote{We wholeheartedly thank the anonymous referee for
noting that Theorem \ref{thm:main} should hold in greater generality.} By decoupling $\rank_B(A)$ from $n$, we define a
notion of \emph{higher Galois closure} $G^{(n)}(A/B)$, see Section
\ref{sub:gcs}; the Galois closure from \cite{bhargava-satriano} is
then given by $G(A/B)=G^{(\rank A)}(A/B)$. In our case of interest, we
show (see Lemma \ref{l:simplifying-the-gc-ideal}) that
\[
G^{(n)}(A_m/\kk)=\kk[x_{i,j}\mid 1\leq i\leq m,1\leq j\leq n] / I
\]
where
\begin{equation}\label{eqn:maineqnI}
  I = (e_1(x_i)\mid 1\leq i\le m) + \sum_{j=1}^n(x_{1,j},\dots,x_{m,j})^2
\end{equation}
and $e_1(x_i):=\sum_j x_{i,j}$. By eliminating $x_{i,n}$, we may naturally view $G^{(n)}(A_m/\kk)$ as a point in $\hilb^{pts}(\bbA^{m(n-1)})$. 

A graded $\kk$-algebra $R$ is said to have trivial negative tangents if the negatively graded pieces of the $T^1$-module
  $T^1(R/\kk,R)_{<0}$ vanish; see Section \ref{sub:tcc}
  for a brief review of the $T^i$-modules. In \cite[Theorem 1.2]{Jelisiejew--2019}, Jelisiejew proves that if $R$ has trivial negative tangents, then all components of the Hilbert scheme containing $[R]$ are elementary. 

In what follows, we work throughout over $\kk$ and so simply write $G^{(n)}(A_m)$ in place of $G^{(n)}(A_m/\kk)$. We prove:

\begin{theorem}\label{thm:main-higher-gc}
$G^{(n)}(A_m)$ has trivial negative tangents if and only if one of the following holds:
\begin{enumerate}
\item\label{neqal1} $n=1$,
\item $m=1$ and $n\leq2$,
\item\label{neq3mgeq3} $n=3$ and $m\geq3$,
\item\label{ngeq4mgeq2} $n\geq4$ and $m\geq2$.
\end{enumerate}
For all $(n,m)$ in \ref{neqal1}--\ref{ngeq4mgeq2} above, every irreducible component of
  $\hilb^{d(n,m)}(\bbA^{m(n-1)})$ containing $G^{(n)}(A_m)$ is
  elementary; see \eqref{eqn:dnm} for a combinatorial formula for $d(n,m)$.
\end{theorem}

\begin{remark}\label{rmk:main-n-equals3}
  In Theorem \ref{thm:n3}, we show that for $n=3$ and $m\geq11$, the
  higher Galois closure $G^{(3)}(A_m)$ lives on an irreducible locus
  $Z_m:=\bbA^{2m}\times\Gr(\binom{m}{2},m^2+m+1)$ whose dimension is
  larger than that of the main component of the Hilbert scheme.
  However, showing that $G^{(3)}(A_m)$ lives on an elementary
  component
  is more involved.  In fact, $Z_m$ is not generally equal to the full component
  containing $G^{(3)}(A_m)$ as one can check from a tangent space
  computation.
\end{remark}

Furthermore, we prove that when $n=3$, the higher Galois closure $G^{(3)}(A_m)$ always defines a smooth point of the Hilbert scheme. We do this by proving vanishing of the obstruction space for deformations. Combining Theorem \ref{thm:main-higher-gc} with Theorems 4.2 and 4.5 of \cite{Jelisiejew--2019}, we see that the obstruction lives in $T^2(G^{(n)}(A_m)/\kk,G^{(n)}(A_m))_{\geq0}$. We characterize precisely when this obstruction space vanishes.

\begin{theorem}\label{thm:VNOS}
Let $G^{(n)}(A_m)$ have trivial negative tangents and assume $G^{(n)}(A_m)\not\simeq\kk$, i.e., $(n,m)$ is as in Theorem \ref{thm:main-higher-gc} \ref{neq3mgeq3}--\ref{ngeq4mgeq2}. Then the obstruction space $T^2(G^{(n)}(A_m)/\kk,G^{(n)}(A_m))_{\geq0}$ vanishes if and only if
\begin{enumerate}[label=(\alph*)]
\item\label{vnos1} $n=3$ and $m\geq3$, or
\item\label{vnos2} $n=4$ and $2\leq m\leq3$.
\end{enumerate}
In particular, for all $(n,m)$ in \ref{vnos1}--\ref{vnos2} above, $G^{(n)}(A_m)$ is a smooth point of $\hilb^{d(n,m)}(\bbA^{m(n-1)})$ living on a unique elementary component.
\end{theorem}

\begin{remark}\label{rmk:vnos-intro}
  One may wonder whether $G^{(n)}(A_m)$ defines a singular point of
  the scheme $\hilb^{d(n,m)}(\bbA^{m(n-1)})$ for $(n,m)$ where the
  obstruction space is non-vanishing, i.e., for those $(n,m)$ not of
  the form \ref{vnos1}--\ref{vnos2}. This appears to be a highly
  subtle question. See Remark \ref{rmk:VNOS} for more details.
\end{remark}

Given the new elementary components we construct, one is particularly
interested in having a formula for $d(n,m):=\dim_\kk G^{(n)}(A_m)$,
and more generally, the Hilbert function of $G^{(n)}(A_m)$. We achieve
this in Theorem \ref{thm:structure-thm-higher-gc} which gives much more
refined information.

As mentioned above, one of the main theorems of \cite{bhargava-satriano} gives an explicit decomposition of $G(A_{n-1})$ into irreducible $S_n$-representations. We generalize this structure theorem to higher Galois closures. This result plays a key role in our proof of Theorem \ref{thm:VNOS} as well as our analysis of the socle of $G^{(n)}(A_m)$. In what follows, $V_\mu$ denotes the irreducible representation (i.e., the \emph{Specht module}) associated to a partition $\mu$, $K_{\mu\lambda}$ is the Kostka number \cite[Definition 2.11.1]{sagan}, and $\vartriangleright$ denotes dominance of partitions.

\begin{theorem}\label{thm:structure-thm-higher-gc}
For $m,n\geq1$, we have an isomorphism of $S_n$-representations
\[
G^{(n)}(A_m)\simeq\bigoplus_{\substack{\mu\vartriangleright\lambda\\ \mu_1=\lambda_1}}\kappa_\lambda K_{\mu\lambda}V_\mu,
\]
where $\lambda=(\lambda_1,\lambda_2,\dots)$ and $\mu=(\mu_1,\mu_2,\dots)$ are weakly decreasing sequences which run over partitions of $n$ into at most $\mathscr{M}:=\min(n,m+1)$ parts, and 
\[
\kappa_\lambda:=\binom{m}{k_0;\dots;k_{\mathscr{M}-1}}
\]
is the multinomial coefficient with $k_j:=\#\{i\neq1 \mid \lambda_i=j\}$ for $j>0$, and $k_0:=m-\sum_{j>0}k_j$. In particular,
\begin{equation}\label{eqn:dnm}
d(n,m)=\sum_{\substack{\mu\vartriangleright\lambda\\ \mu_1=\lambda_1}}\kappa_\lambda K_{\mu\lambda}\dim V_\mu
\end{equation}
and the Hilbert function of $G^{(n)}(A_m)$ is given by
\begin{equation}\label{eqn:hnm}
h_{G^{(n)}(A_m)}(i)=\sum_{\substack{\mu\vartriangleright\lambda\\ \mu_1=\lambda_1=n-i}}\kappa_\lambda K_{\mu\lambda}\dim V_\mu.
\end{equation}
\end{theorem}

In an earlier paper \cite{Satriano--Staal--2023}, we found 
families of small dimensional elementary components of Hilbert schemes
of points in $\bbA^n$, answering a question of Iarrobino.  One of the
new insights in that paper was that we were able to produce secondary
families of elementary components by enlarging our ideals by socle
elements. In this current paper, we show that this socle phenomenon is not
an isolated occurrance, and indeed holds for the families we construct
in Theorem \ref{thm:main-higher-gc}.  Specifically, we prove that if
$s_1,\dots,s_r\in G^{(n)}(A_m)$ are socle elements that avoid specific
multidegrees, then the quotient $G^{(n)}(A_m)/(s_1,\dots,s_r)$ also
lives on an elementary component of $\hilb^{pts}(\bbA^{m(n-1)})$.

In order to state this result precisely, we recall from \eqref{eqn:maineqnI} that $G^{(n)}(A_m)$ carries a $\ZZ^m$-grading (and hence a $\ZZ$-grading), where $x_{i,j}$ has multidegree $(0,\dots,0,1,0,\dots,0)$ with the $1$ in the $i$th position. Our next main result is the following:

\begin{theorem}\label{thm:main-higher-gc-socle}
Let $n \geq 4$ and $m \geq 2$. Choose $s_1, \dots, s_r$ linearly independent homogeneous socle elements of minimal degree $D$, and assume $D \geq 3$. If $D \in \{\lfloor\frac{n}{2}\rfloor, \lceil\frac{n}{2}\rceil\}$, then let $s_i$ have multidegree $d_i := (d_{i,1},...,d_{i,m})$ and assume each $d_{i,j}\leq D-2$. If
\[
B:=G^{(n)}(A_m)/(s_1,\dots,s_r)
\]
has no socle elements of degree strictly less than $D$, then $B$ has trivial negative tangents. In particular, every irreducible component of $\hilb^{d(n,m)-r}(\bbA^{m(n-1)})$ containing $B$ is elementary.

Furthermore, if $m\geq n-1$ and either:
\begin{enumerate}
\item\label{socleeven} $n$ is even, or
\item\label{socleodd} $n$ is odd and $B$ has no socle elements of degree $D-1$,
\end{enumerate}
then $B$ automatically has no socle elements of degree strictly less than $D$, and hence has trivial negative tangents.
\end{theorem}

\begin{remark}
The condition that $B$ have no socle elemens in degree $D-1$ in Theorem \ref{thm:main-higher-gc-socle}\ref{socleodd} is necessary, see Examples \ref{eg:n-equals-5} and \ref{eg:n-equals-7}. This pair of examples shows that this condition 
is sometimes satisfied and sometimes not.
\end{remark}

In Theorem \ref{thm:socles-gc}, we analyze the $S_n$-representation structure of the socle of $G^{(n)}(A_m)$. Combining this analysis with Theorem \ref{thm:main-higher-gc-socle}\ref{socleeven}, we give a combinatorial formula for the range of values that $r$ may assume. This formula is given in Corollary \ref{cor:main-higher-gc-socle} below.

Let $n$ be even and let $\cP$ be the set of partitions $\lambda$ of $n$ satisfying the
following constraints:~$\lambda=(\lambda_1,\dots,\lambda_\ell)$ with
$\lambda_1=\frac{n}{2}$,
$\lambda_2\leq\frac{n}{2}-2$, and
$\lambda_1\geq\lambda_2\geq\dots\geq\lambda_\ell>0$.
Let $C_k$ denote the $k$th Catalan number and let
\[
R(n):=C_{n/2}\sum_{\lambda\in\cP}\kappa_\lambda 
\]
We then have:
\begin{corollary}\label{cor:main-higher-gc-socle}
  Let $n\geq6$ be even and let $m\geq n-1$. For all $0\leq r\leq R(n)$, there exist socle elements $s_1,\dots,s_r\in G^{(n)}(A_m)$ such that $B:=G^{(n)}(A_m)/(s_1,\dots,s_r)$ has trivial negative tangents. In particular, every irreducible component of $\hilb^{d(n,m)-r}(\bbA^{m(n-1)})$ containing $B$ is elementary.
\end{corollary}

Combining our formula \eqref{eqn:hnm} for the Hilbert function of
$G^{(n)}(A_m)$ with Theorem \ref{thm:main-higher-gc-socle} and Corollary
\ref{cor:main-higher-gc-socle}, we obtain the following tables (for small $n,m$) recording the Hilbert functions of our algebras $G^{(n)}(A_m)/(s_1,\dots,s_r)$ with trivial negative tangents. Table \ref{hf} is generated from Theorem \ref{thm:main-higher-gc-socle} and computer calculations, and Table \ref{hfcor19} is an immediate consequence of Corollary \ref{cor:main-higher-gc-socle}.

\begin{table}[hb]
  \caption{Hilbert Functions of $G^{(n)}(A_m) / (s_1,\dotsc,s_r)$
    with trivial negative tangents for small $n,m$ \label{hf}}
  \centering
  \resizebox{\columnwidth}{!}{%
  {\renewcommand{\arraystretch}{1.5}
  \begin{tabular}{|c||l|l|l|l|}
    \hline    
    $n$ & $G^{(n)}(A_2)$ & $G^{(n)}(A_3)$ & $G^{(n)}(A_4)$ &
    $G^{(n)}(A_5)$ \\ \hline\hline 
    $3$ & N/A & $(1,6,3)$ & $(1,8,6-r)$, $0\le r\le 2$ &
    $(1,10,10-r)$, $0\le r\le 5$ \\ \hline
    $4$ & $(1,6,9-r)$, $0\le r\le 5$ & $(1,9,21-r,1)$, $0\le r\le 12$
    & $(1,12,38-r,4)$, $0\le r\le 20$ & $(1,15,60-r,10)$, $0\le r\le
    30$ \\ \hline
    $5$ & $(1,8,21,10-r)$, $0\le r\le 6$ & $(1,12,48,44-r)$, $0\le
    r\le 27$ & $(1,16,86,116-r,1)$, $0\le r\le 40$ &
    $(1,20,135,240-r,5)$, $0\le r \le 100$ \\ \hline
  \end{tabular}}
  }
\end{table}

\begin{table}[hb]
  \caption{Examples of Hilbert functions of $G^{(n)}(A_m) / (s_1,\dotsc,s_r)$
    with trivial negative tangents derived from Corollary \ref{cor:main-higher-gc-socle}\label{hfcor19}}
  \centering
  \resizebox{\columnwidth}{!}{%
  {\renewcommand{\arraystretch}{1.5}
  \begin{tabular}{|c||l|l|l|l|}
    \hline    
    $n$ & $G^{(n)}(A_{n-1})$ & $G^{(n)}(A_{n})$ & $G^{(n)}(A_{n+1})$ 
 \\ \hline\hline 
    $6$ 
    & $(1, 25, 235, 915-r, 680, 1)$, $0\leq r\leq 50$ 
    & $(1, 30, 339, 1600-r, 1545, 6)$, $0\leq r\leq 100$ 
    & $(1, 35, 462, 2562-r, 3045, 21)$, $0\leq r\leq 175$ 
     \\ \hline
    $8$ 
    & $(1, 49, 1001, 10745, 60501-r, 128457, 26173, 1)$, $0\leq r\leq 2254$ 
    & $(1, 56, 1308, 16072, 104006-r, 259224, 67396, 8)$, $0\leq r\leq 3724$ 
    & $(1, 63, 1656, 22920, 167580-r, 479430, 152124, 36)$, $0\leq r\leq 5796$ 
     \\ \hline
  \end{tabular}}
  }
\end{table}


\begin{remark}
  Observe that the cases $(n,m)=(3,m)$, for
  $3\le m\le 5$, and $(n,m) = (4,2)$ all define algebras of nilpotency
  class $2$, in the sense of \cite{Shafarevich--1990}. We therefore recover Shafarevich's 
examples with Hilbert functions $(1,d,e)$, for $3\le e\le  \frac{(d-1)(d-2)}{6}+2$. 
Note that our examples $(1,6,9-r)$, for $0\le r\le 3$ 
 provide new data for
  Shafarevich's unexplored ``middle interval'' $\frac{(d-1)(d-2)}{6}+2
  < e < \frac{d^2-1}{3}$ in the specific case $d=6$ (see
  \cite[p.179]{Shafarevich--1990}). 
\end{remark}

Notice that $G^{(4)}(A_3)$ modulo all of its socle elements of degree
$2$ is Gorenstein local with Hilbert function $(1,9,9,1)$, and thus
recovers an elementary component found by Iarrobino and Kanev.  That
is, this algebra is compressed (see \cite{Iarrobino--1984},
\cite[Section 3.3]{Iarrobino--Emsalem--1978}), and is generic
nonsmoothable by \cite[Lemma 6.21]{Iarrobino--Kanev--1999}, meaning it
sits in the Hilbert scheme on a single component consisting
generically of Gorenstein algebras with the same Hilbert function.

\vspace{1.2em}

In light of Theorem \ref{thm:main-higher-gc}, we end the introduction
by posing the following general question.

\begin{question}\label{q:higher-gc-TNT}
  For which $n$ and which finite rank $\kk$-algebras $A$ does the
  higher Galois closure $G^{(n)}(A/\kk)$ have trivial negative tangents, or more generally, lie on an elementary component?
\end{question}
 For small $n$, Table \ref{table} (in Section \ref{sec:last}) lists those $G^{(n)}(A/\kk)$ with trivial negative tangents for each isomorphism class of $\kk$-algebras of rank at most $m$. This makes use of Poonen's work \cite{Poonen--2008--iso}, which 
lists representatives of every such isomorphism class. 
It would be interesting to answer Question \ref{q:higher-gc-TNT} even for these finitely many classes.  

\subsection*{Acknowledgments}
It is a pleasure to thank Joachim Jelisiejew for comments on an
earlier draft of this paper. We are deeply indebted to the anonymous referee whose suggestions both shortened our arguments and generalized our results. The referee greatly simplified our proofs of Propositions 
\ref{prop:degneg1-second-syzygy-generaln} and \ref{prop:GnTNT-evenImpliesOdd}, and noted that our results should extend to the case when $m\neq n-1$.

\section{Preliminaries}\label{sec:preliminaries}

\subsection{Review of the Truncated Cotangent Complex}
\label{sub:tcc}

We prove Theorems \ref{thm:main}, \ref{thm:main-higher-gc}, and \ref{thm:main-higher-gc-socle} by showing that
our algebras of interest have trivial negative tangents. In this
section, we give a brief review of the truncated cotangent complex and
the $T^i$-modules. For further details, we refer the reader to
\cite[\S3]{Hartshorne--2010} or
\cite{Lichtenbaum--Schlessinger--1967}.

Given any morphism $B\to A$ of rings, we obtain a model of the
truncated cotangent complex $L_{A/B,\bullet}$ as follows. Let
$\pi\colon R\to A$ be a surjective map from a (possibly infinite type)
polynomial ring $R$ over $B$. Let $I=\ker(\pi)$, $F$ be a (possibly
infinite rank) free $R$-module, and $\pi'\colon F\to I$ be a
surjective $R$-module map. Let $Q=\ker(\pi')$ and $\Kos\subseteq Q$ be
the submodule of Koszul relations. The \emph{\bfseries truncated
cotangent complex} of $A$ over $B$ is the following $3$-term complex
concentrated in homological degrees $0,1,2$:
\[
L_{A/B,\bullet} \colon \Omega^1_{R/B} \otimes_{R} A
\stackrel{\ d_1}{\longleftarrow} F \otimes_{R} A
\stackrel{\ d_2}{\longleftarrow} Q/\Kos;
\]
the map $d_2$ is induced by the inclusion $Q \subseteq F$ and $d_1$ is
given by composing the map $F \otimes_{R} A \to I/I^2$ with the map
induced by the derivation $R \to \Omega^1_{R/B}$. One then defines the
\emph{\bfseries $T^i$-modules} to be
\[
T^i(A/B, M) := H^i(\Hom_A(L_{A/B,\bullet}, M)),
\]
for any $A$-module $M$ and $0 \le i \le 2$. Different choices of $R$,
$F$, $\pi$, and $\pi'$ yield quasi-isomorphic elements
$L_{A/B,\bullet}$ of the derived category, see e.g.\ \cite[Remark
  3.3.1]{Hartshorne--2010}; as a result, the $T^i$-modules depend only
on $B\to A$ and $M$. Furthermore, when $A$ and $B$ carry gradings by
an abelian group and $B\to A$ is a graded morphism, then the choices
of $R$, $F$, $\pi$, and $\pi'$ can be made to respect the grading. As
a result, if $M$ is a graded $A$-module, then $T^i(A/B,M)$ is also
graded and the nine-term long exact sequences given in \cite[Theorems
  3.4--3.5]{Hartshorne--2010} are graded morphisms.

Lastly, if $\kk$ is a field, then a (positively) $\ZZ$-graded
$\kk$-algebra $A$ is said to have \emph{\bfseries trivial negative
tangents} if the negatively graded pieces $T^1(A/\kk, A)_{<0}$ vanish. Our
primary case of interest is where $A=S/I$ with $S=\kk[x_1,\dots,x_n]$,
$I$ is a homogeneous ideal, and each variable has degree $1$. Then
differentiation by $x_i$ yields an $S$-module map $\partial_i\colon
I\to A$. In this case, $A$ has trivial negative tangents if and only
if every negatively graded $S$-module map $\varphi\colon I\to A$ is a
$\kk$-linear combination of $\partial_1, \dotsc, \partial_n$.

\subsection{Galois closures and higher Galois closures}
\label{sub:gcs}

We briefly review the Galois closure operation introduced in
\cite{bhargava-satriano} and then introduce our \emph{higher Galois closure} construction. Let $A$ be a (commutative) $B$-algebra which
is free of rank $n<\infty$ as a $B$-module. (The Galois closure is, in
fact, defined whenever $A$ is locally free of rank $n$, however we do
not need this level of generality here.) For $a\in A$ and $1\leq i\leq
n$, let $a^{(i)}=1\otimes\dotsb\otimes 1\otimes
a\otimes1\otimes\dotsb\otimes1\in A^{\otimes n}$, where the $a$ is in
the $i$th position and the tensor product is taken over $B$. The
intuition behind the Galois closure is to treat the elements
$a^{(1)},\dotsc,a^{(n)}$ as Galois conjugates. In particular, the
elementary symmetric functions in the $a^{(i)}$ should be expressible
as coefficients of an appropriate characteristic polynomial, as we now
describe.

For every $a\in A$, we obtain a $B$-linear map $m_a\colon A\to A$
given by multiplication by $a$. Let $p_a(t)\in B[t]$ be the
characteristic polynomial of $m_a$. Write $p_a(t)=\sum_{i=0}^n
(-1)^is_i(a)t^{n-i}$ and let $e_j(a)$ be the $j$th elementary
symmetric function in $a^{(1)},\dots,a^{(n)}$. Then
\[
G(A/B):=A^{\otimes n}/I,\quad\textrm{where}\quad
I=\llrr{s_j(a)-e_j(a)\mid a\in A}.
\]
The $S_n$-action on $A^{\otimes n}$ given by permuting tensor factors
descends to an action on $G(A/B)$. If $a_1,\dots,a_n$ is a basis for
$A$ as a $B$-module, then $I$ is generated by the expressions
$s_j(a_i)-e_j(a_i)$ for $1\leq i,j\leq n$, see
\cite[\S2]{bhargava-satriano}.

We define the higher Galois closure by allowing $n$ to be independent of $\rank_B(A)$. 
We simply let
\[
G^{(n)}(A/B):=A^{\otimes n}/I,\quad\textrm{where}\quad
I=\llrr{s_j(a)-e_j(a)\mid a\in A}.
\]
Specializing to our case of interest, let
$A_m=\kk[x_1,\dots,x_m]/(x_1,\dots,x_m)^2$ where $\kk$ is a
characteristic $0$ field. Let $x_{i,j}:=x_i^{(j)}$, and for any linear form $g=\sum_{i=1}^m\lambda_i x_i$ with $\lambda_i\in\kk$, let $g_j:=\sum_{i=1}^m\lambda_i x_{i,j}$ and $e_\ell(g):=e_\ell(g_1,\dots,g_n)$ denote the $\ell$th elementary symmetric function. We see then that
\begin{equation}\label{eqn:gc-def}
G^{(n)}(A_m/\kk)=\kk[x_{i,j}\mid 1\leq i\leq m,1\leq j\leq n]/I
\end{equation}
with
\[
I=\sum_{j=1}^n(x_{1,j},\dots,x_{m,j})^2+\llrr{e_1(g),\dots,e_n(g)}
\]
where $g$ runs through all linear forms in the $x_i$. In this notation, the $S_n$-action on
$G^{(n)}(A_m/\kk)$ is given by $\sigma(x_{i,j})=x_{i,\sigma(j)}$.

\begin{lemma}\label{l:simplifying-the-gc-ideal}
We have
\[
I=\bigl(e_1(x_i)\mid 1\leq i\leq m\bigr)+\sum_{j=1}^n(x_{1,j},\dots,x_{m,j})^2.
\]
\end{lemma}
\begin{proof}
We prove that $I=\llrr{e_1(g)\mid g\textrm{\ linear\ form\ in\ the\ } x_i}+\sum_{j=1}^n(x_{1,j},\dots,x_{m,j})^2$. The result follows since $\llrr{e_1(g)\mid g\textrm{\ linear\ form\ in\ the\ } x_i}=(e_1(x_i)\mid 1\leq i\leq m)$.

We first note that all power sums
$P_\ell(g):=P_\ell(g_1,\dots,g_n):=\sum_j g_j^\ell$ are contained in
$I$; indeed, $P_1(g)$ is a linear combination of the $e_1(x_i)$ which are in $I$; 
for $\ell>1$, the power sum $P_\ell(g)$ is in the
ideal generated by the $g_j^2$, which is contained in $(x_{1,j},\dots,x_{m,j})^2$. The Newton--Girard identities then
express each $e_k(g)$ as a $\QQ$-linear
combination of the $P_\ell(g)$. Since
$\QQ\subset\kk$, we have finished the proof.
\end{proof}

\section{Structure theorem for higher Galois closures: Theorem \ref{thm:structure-thm-higher-gc}}
\label{sec:highergcstructure}

In this section, we prove Theorem \ref{thm:structure-thm-higher-gc} which generalizes one of the main results of \cite{bhargava-satriano} to higher Galois closures. This will play an important role in Sections \ref{sec:socles-gc} and \ref{sec:vnos}, where we analyze the socle and obstruction space of $G^{(n)}(A_m)$.

The following lemma will be useful throughout this section as well as Section \ref{sec:socles-gc}.

\begin{lemma}
\label{l:multidegleqnparts}
Let $m\leq m'$, $d=(d_1,\dots,d_m)\in\NN^m$, and $d'=(d_1,\dots,d_m,0,\dots,0)\in\NN^{m'}$. Then the natural projection map on multigraded pieces
\[
G^{(n)}(A_{m'})_{d'}\stackrel{\simeq}{\longrightarrow} G^{(n)}(A_m)_{d}
\]
is an isomorphism of $S_n$-representations.
\end{lemma}
\begin{proof}
Letting $I_{n,m}\subset A_m^{\otimes n}$ denote the defining ideal for $G^{(n)}(A_m)$, simply observe that we have a commutative diagram
\[
\xymatrix{
0\ar[r] & (I_{n,m'})_{d'}\ar[r]\ar[d]_-{\simeq} & (A_{m'}^{\otimes n})_{d'}\ar[r]\ar[d]_-{\simeq} & G^{(n)}(A_{m'})_{d'} \ar[d]\ar[r] & 0\\
0\ar[r] & (I_{n,m})_{d}\ar[r] & (A_m^{\otimes n})_{d}\ar[r] & G^{(n)}(A_m)_{d}\ar[r] & 0
}
\]
where the rows are exact and the leftmost two vertical maps are isomorphisms of $S_n$-representations.
\end{proof}

\begin{lemma}
\label{l:lemma31}
Let $d=(d_1,\dots,d_m)\in\NN^m$.
\begin{enumerate}
\item\label{lemma31part} If $n-\sum_i d_i < d_j$ for some $j$, then $G^{(n)}(A_m)_d=0$.
\item\label{supportSizeLeqn-1} If $|\{i\mid d_i\neq0\}|\geq n$, then $G^{(n)}(A_m)_d=0$.
\end{enumerate}
\end{lemma}
\begin{proof}
Lemma 31 of \cite{bhargava-satriano} says that \ref{lemma31part} holds when $m=n-1$, however the proof applies verbatim when $m$ is arbitrary:~the proof uses an inclusion-exclusion argument on indices $j$, which is applicable in our case as $x_{i,j}x_{k,j}=0$.

For \ref{supportSizeLeqn-1}, we first notice that if $|\{i\mid d_i\neq0\}|>n$, then $G^{(n)}(A_m)_d=0$ by the pigeonhole principle:~for every monomial $g\in G^{(n)}(A_m)_d$, there exists $j$ and $i\neq k$ such that $x_{i,j}x_{k,j}$ divides $g$. We may therefore assume $|\{i\mid d_i\neq0\}|=n$. By \ref{lemma31part}, after reordering variables, we know $d=(1,\dots,1,0,\dots,0)$. Then applying Lemma \ref{l:multidegleqnparts}, we see $G^{(n)}(A_m)_d\simeq G^{(n)}(A_n)_{(1,\dots,1)}$ and the latter vanishes again by \ref{lemma31part}.
\end{proof}

We now turn to Theorem \ref{thm:structure-thm-higher-gc}.

\begin{proof}[{Proof of Theorem \ref{thm:structure-thm-higher-gc}}]
We have a decomposition
\[
G^{(n)}(A_m)=\bigoplus_{d\in\ZZ^m}G^{(n)}(A_m)_d
\]
as $S_n$-representations. To determine $G^{(n)}(A_m)_d$ as a representation, we may reorder variables in order to assume $d=(d_1,\dotsc,d_m)$ with $n-\sum_id_i\geq d_1\geq\dotsb\geq d_m$. If $m\leq n-1$, then let $d'=(d_1,\dotsc,d_m,0,\dots,0)\in\NN^{n-1}$ and, applying Lemma \ref{l:multidegleqnparts}, we have an isomorphism
\[
G^{(n)}(A_{n-1})_{d'}\stackrel{\simeq}{\longrightarrow} G^{(n)}(A_m)_d.
\]
If $m\geq n$, then by Lemma \ref{l:lemma31}\ref{supportSizeLeqn-1}, we know $d_i=0$ for $i\geq n$ and so we may identify $d$ with $d'=(d_1,\dots,d_{n-1})\in\NN^{n-1}$. Then again applying Lemma \ref{l:multidegleqnparts}, we have an isomorphism
\[
G^{(n)}(A_m)_d\stackrel{\simeq}{\longrightarrow} G^{(n)}(A_{n-1})_{d'}.
\]
Let $\lambda$ be the (ordered) partition of $n$ given by $(n-\sum_i d_i, d'_1, \dots,d'_{n-1})$; we refer to $\lambda$ as the partition associated to $d$. Theorem 27 of \cite{bhargava-satriano} shows
\[
G^{(n)}(A_{n-1})_{d'}=\bigoplus_{\substack{\mu\vartriangleright\lambda\\ \mu_1=\lambda_1}}K_{\mu\lambda}V_\mu.
\]
Note that after removing zero entries from $\lambda$, the result is a
partition of $n$ into at most $\mathscr{M}:=\min(n,m+1)$
parts. Lastly, for any such partition $\lambda$, the number of
$d\in\ZZ^m$ associated to $\lambda$ is given by the multinomial
coefficient $\binom{m}{k_0;\dots;k_{n-1}}$ where $k_j:=\#\{i\neq1 \mid
\lambda_i=j\}$ for $j>0$, and $k_0:=m-\sum_{j>0}k_j$.
\end{proof}

\section{Trivial negative tangents for Galois closures}
\label{sec:triviality-neg-tngts}

We use the notation from equation \eqref{eqn:gc-def} where $I$
is given as in Lemma \ref{l:simplifying-the-gc-ideal}. We give
$G(A_n)$ the $\ZZ^m$-grading where $x_{i,j}$ has degree
$(0,\dots,0,1,0,\dots,0)$ with $1$ in the $i$th place. This induces
the $\ZZ$-grading where every $x_{i,j}$ has degree $1$.

Throughout the rest of this section, we
let $\mathscr{R}=\kk[x_{i,j}\mid 1\leq i\leq m,\,1\leq j\leq n]$ and fix a
graded $\mathscr{R}$-module map
\[
\varphi\colon I\to G^{(n)}(A_m)
\]
of negative degree. The primary goal of this section is to prove
$G^{(n)}(A_m)$ has trivial negative tangents. To do so, we must show that
$\varphi$ is a $\kk$-linear combination of the maps $\partial_{i,j}$
given by differentiation by $x_{i,j}$.

We prove Theorem \ref{thm:main-higher-gc} by first showing the result for $n\geq4$ even. We then reduce the case where $n$ is odd to the case $n-1$. Since $G^{(2)}(A_m)$ does not generally have trivial negative tangents, we must handle the case $n=3$ separately.

\subsection{Tangents of degree at most minus two}
\label{subsec:degnegleqmin2}


\begin{proposition}\label{prop:no-degleqneg2-generalcase}
  Let $n,m\geq2$. If $\varphi\colon I\to G^{(n)}(A_m)$ has degree at most $-2$, then
  $\varphi=0$.
\end{proposition}
\begin{proof}
  By Lemma \ref{l:simplifying-the-gc-ideal}, $I$ is generated by
  quadratic and linear forms. As a result, if $\deg(\varphi)<-2$, then
  $\varphi=0$, so we need only consider the case where
  $\deg(\varphi)=-2$. Then $\varphi$ sends the quadratic generators of
  $I$ to constants, and for $i\neq k$, we have
  \begin{equation}\label{eqn:syzygy-used-to-show-deg--2-vanishes}
    x_{k,j}\varphi(x_{i,j}^2)=x_{i,j}\varphi(x_{i,j}x_{k,j}).
  \end{equation}
  However, $x_{1,j},\dots,x_{m,j}$
  are linearly independent elements of $G^{(n)}(A_m)$, and hence,
  $\varphi(x_{i,j}^2)=\varphi(x_{i,j}x_{k,j})=0$.
\end{proof}

\subsection{Tangents of degree minus one:~preliminaries}

Having dispensed with the case where $\deg(\varphi)\leq -2$, we must
now consider the case where $\deg(\varphi)=-1$. Showing triviality of
such tangents will involve a careful analysis of the syzygies of $I$,
and occupies the remainder of Section \ref{sec:triviality-neg-tngts}.

We begin by obtaining some preliminary information about the images
under $\varphi$ of the quadratic generators of $I$.

\begin{lemma}\label{l:degneg1-obvious-syzygy-generalcase}
  Let $\varphi\colon I\to G^{(n)}(A_m)$ have degree $-1$. Assume either (i) $n\geq4$ and $m\geq2$ or (ii) $n=3$ and $m\geq3$. Then
  $\varphi(x_{i,j}^2)$ and $\varphi(x_{i,j}x_{k,j})$ are in the
  $\kk$-linear span of $x_{1,j},\dotsc,x_{n-1,j}$.
\end{lemma}
\begin{proof}
  Let $m\geq2$.  By symmetry and for ease of notation, it is enough to
  show the statement for $\varphi(x_{1,1}^2)$ and
  $\varphi(x_{1,1}x_{2,1})$. By degree constraints we know
  $\varphi(x_{1,1}^2)$ and $\varphi(x_{1,1}x_{2,1})$ are in the
  $\kk$-span of the $x_{i,j}$. Using that $e_1(x_i)=0$ in
  $G^{(n)}(A_m)$, we may write
  \[
  \varphi(x_{1,1}^2)=\sum_i\sum_{j\neq
    n}a_{i,j}x_{i,j}\quad\textrm{and}\quad
  \varphi(x_{1,1}x_{2,1})=\sum_i\sum_{j\neq n}b_{i,j}x_{i,j},
  \]
  where $a_{i,j},b_{i,j}\in \kk$. Consider the syzygy
  \begin{equation}\label{eqn:first-quadratic-syzgy}
    x_{2,1}\varphi(x_{1,1}^2)=x_{1,1}\varphi(x_{1,1}x_{2,1}).
  \end{equation}
  Expanding, we see
  \[
  \sum_i\sum_{j\neq1,n}a_{i,j}x_{2,1}x_{i,j} =
  \sum_i\sum_{j\neq1,n}b_{i,j}x_{1,1}x_{i,j}.
  \]
  We must show that $a_{i,j}=b_{i,j}=0$ for $j\neq1,n$.

  By considering the $\ZZ^m$-grading on $G^{(n)}(A_m)$, we see
  \begin{equation}\label{eqn:n-2-2}
    \sum_{j\neq1,n}a_{1,j}x_{2,1}x_{1,j} =
    \sum_{j\neq1,n}b_{2,j}x_{1,1}x_{2,j};
  \end{equation}
  for $i\neq 2$ we have
  \begin{equation}\label{eqn:n-2-1-1-first}
    \sum_{j\neq1,n}b_{i,j}x_{1,1}x_{i,j}=0
  \end{equation}
  and for $i\neq 1$ we have
  \begin{equation}\label{eqn:n-2-1-1-second}
    \sum_{j\neq1,n}a_{i,j}x_{2,1}x_{i,j}=0.
  \end{equation}
  First assume $n\geq4$. We claim it is enough to prove linear
  independence of the set
  $\mathcal{S}:=\{x_{1,1}x_{2,j},x_{2,1}x_{1,j}\mid
  j\neq1,n\}$. Indeed, this would show that all coefficients $a_{1,j}$
  and $b_{2,j}$ vanish in equation \eqref{eqn:n-2-2}; furthermore, the
  transposition $(2,i)\in S_n$ acts on $G^{(n)}(A_m)$, sending
  $x_{1,1}x_{i,j}$ to $x_{1,1}x_{2,j}$, so linear independence of
  $\mathcal{S}$ implies the vanishing of all coefficients $b_{i,j}$ in
  equation \eqref{eqn:n-2-1-1-first}. A similar argument applies for
  equation \eqref{eqn:n-2-1-1-second}.

  Let $G^{(n)}(A_m)_2$ denote the $\kk$-vector space of elements of $G^{(n)}(A_m)$
  whose total degree is $2$; this is the vector space generated by
  elements of the form $x_{k,\ell}x_{i,j}$. To prove linear
  independence of $\mathcal{S}$, by symmetry of $x_1$ and $x_2$, it is
  enough show that for every $1<j_0<n$, there is a linear functional
  $f_{j_0}$ on $G^{(n)}(A_m)_2$ such that for all $1<j<n$, we have
  $f_{j_0}(x_{2,1}x_{1,j})=0$ and $f_{j_0}(x_{1,1}x_{2,j})=\delta_{j,j_0}$
  with $\delta$ the Kronecker delta function. Since $n\geq4$, we can
  choose $\ell_0\neq 1,j_0,n$. Then we define
  \[
  f_{j_0}(x_{1,1}x_{2,j_0})=f_{j_0}(x_{1,\ell_0}x_{2,n})=1,\quad
  f_{j_0}(x_{1,1}x_{2,n})=f_{j_0}(x_{1,\ell_0}x_{2,j_0})=-1,
  \]
  and let $f_{j_0}(x_{k,\ell}x_{i,j})=0$ for all other quadratics with
  $1\leq i,k<n$ and $1\leq \ell, j\leq n$. The functional $f_{j_0}$ is
  well-defined since it vanishes on all expressions of the form
  $x_{k,j}x_{i,j}$ and $x_{k,\ell}e_1(x_i)$ with $1\leq i,k<n$ and
  $1\leq \ell, j\leq n$.

When $n=3$, the set $\mathcal{S}$ is no longer linearly independent. Eliminating the $x_{i,3}$ variables, we see
\begin{equation}\label{eqn:G3Am}
G^{(3)}(A_m)\simeq k[x_{i,1},x_{i,2}\mid 1\leq i\leq m]/(x_{i,j}x_{k,j}, x_{i,1}x_{k,2}+x_{k,1}x_{i,2}).
\end{equation}
Thus, equations \eqref{eqn:n-2-1-1-first} and \eqref{eqn:n-2-1-1-second} tell us $a_{i,2}=b_{i,2}=0$ for $i>2$. Equation \eqref{eqn:n-2-2} tells us $a_{1,2}=-b_{2,2}$. By symmetry of the $x_{i,2}$ variables, we have therefore shown (with a slight change in notation) that
\[
\varphi(x_{1,1}x_{k,1})=\sum_i a_i^{(k)}x_{i,1} + b_1^{(k)}x_{1,2} + b_k^{(k)}x_{k,2}
\]
where $b_k^{(k)}=-2b_1^{(1)}$. Then assuming $m\geq3$, we may consider the syzygy
\[
x_{3,1}\varphi(x_{1,1}x_{2,1})=x_{2,1}\varphi(x_{1,1}x_{3,1}),
\]
we find
\[
b_1^{(2)}x_{1,2}x_{3,1}+b_2^{(2)}x_{2,2}x_{3,1}=b_1^{(3)}x_{2,1}x_{1,2}+b_3^{(3)}x_{2,1}x_{3,2}
\]
which implies $b_1^{(2)}=b_1^{(3)}=0$ and $b_2^{(2)}=-b_3^{(3)}$. Using that $b_k^{(k)}=-2b_1^{(1)}$, we see $b_2^{(2)}=b_3^{(3)}=0$. Hence, $\varphi(x_{1,1}x_{2,1})$ is in the span of the $x_{i,1}$.
\end{proof}


\subsection{The case where \texorpdfstring{$m\geq2$}{mge2} and
  \texorpdfstring{$n\geq4$}{nEven} is even}
\label{subsec:n-even-first-syz}

The proof of Lemma \ref{l:degneg1-obvious-syzygy-generalcase} was
obtained by considering syzygies among the quadratic generators of
$I$. Next we consider syzygies involving quadratic generators as well
as the linear generators $e_1(x_i)$. The form of these syzygies will
depend on the parity of $n$. This subsection considers the case where
$n$ is even; the case where $n$ is odd is handled in
\S\ref{subsec:n-odd-first-syz}.

In Proposition \ref{prop:degneg1-e1-syzygy-general-case-n-even}, we
will obtain refined information about $\varphi(e_1(x_i))$ and
$\varphi(x_{i,j}^2)$. The proof will rely on certain non-vanishing
results which we establish in Lemmas
\ref{l:degneg1-e1-syzygy-general-case--helper} and
\ref{l:degneg1-e1-syzygy-general-case--helper2-n-even}.

\begin{lemma}\label{l:degneg1-e1-syzygy-general-case--helper}
  Let $M\geq1$ and
  \[
  R=\kk[w_1,\dots,w_{2M},z]/\Bigl(w_1^2,\dotsc,w_{2M}^2, \sum_i
  w_i,z^2\Bigr).
  \]
  Then
  \[
  z\prod_{i=1}^M(w_{2i-1}-w_{2i})
  \]
  is a non-zero element of $R$.
\end{lemma}
\begin{proof}
  Consider the quotient $R'=R/J$ where
  $J=(w_1+w_2,\dots,w_{2M-1}+w_{2M})$. Then
  $f:=z\prod_{i=1}^M(w_{2i-1}-w_{2i})$ has image in
  \[
  R'\simeq \kk[w_1,w_3,\dots,w_{2M-1},z] /
  (z^2,w_1^2,w_3^2,\dots,w_{2M-1}^2)
  \]
  given by $2^Mzw_1w_3\dots w_{2M-1}$. Since this expression is
  square-free, it is non-zero in $R'$ and hence $f$ is non-zero in
  $R$.
\end{proof}

\begin{lemma}\label{l:degneg1-e1-syzygy-general-case--helper2-n-even}
  Let $M\geq1$ and
  \[
  R=\kk[w_1,\dots,w_{2M}]/\Bigl(w_1^2,\dotsc,w_{2M}^2,\sum_i
  w_i\Bigr).
  \]
  If $c,a_1,\dots,a_{2M}\in \kk$ and
  \[
  c\prod_{i=1}^M(w_{2i-1}-w_{2i}) =
  \sum_{\substack{j\mathrm{\ odd} \\ 1\leq j\leq 2M}}(a_jw_j-a_{j+1}w_{j+1})
  \prod_{\substack{k\mathrm{\ odd} \\ k\neq j}}(w_k-w_{k+1}),
  \]
  then $c=\frac{1}{2}\sum_{j=1}^{2M}a_j$ holds.
\end{lemma}
\begin{proof}
  Consider the quotient $R'=R/J$ where
  $J=(w_1+w_2,\dots,w_{2m-1}+w_{2m})$. Then
  \[
  R'\simeq \kk[w_1,w_3,\dots,w_{2m-1}]/(w_1^2,w_3^2,\dots,w_{2m-1}^2).
  \]
  In this quotient ring, we have
  \[
  2^mc w_1w_3\dots w_{2m-1}=2^{m-1}\sum_{j=1}^{2m}a_j w_1w_3\dots
  w_{2m-1}.
  \]
  Since $w_1w_3\dots w_{2m-1}$ is square-free, it is non-zero and so
  $c=\frac{1}{2}\sum_{j=1}^{2m}a_j$.
\end{proof}

The following result greatly constrains the form $\varphi$ may take
when $n$ is even.

\begin{proposition}\label{prop:degneg1-e1-syzygy-general-case-n-even}
Let $m\geq2$ and assume $n\geq4$ is even.  If $\varphi\colon I\to G^{(n)}(A_m)$ has degree $-1$,
  then there exist $a_{i,j}\in\kk$ for which
  \[
  \varphi(x_{i,j}^2)=a_{i,j}x_{i,j}\quad\quad\textrm{and}\quad\quad
  \varphi(e_1(x_i))=\frac{1}{2}\sum_j a_{i,j}.
  \]
Furthermore, for any fixed $j$ and any linear form $g=\sum_i \lambda_ix_{i,j}$ with $\lambda_i\in\kk$, we have
\[
\varphi(g^2)=\lambda g
\]
for some $\lambda\in\kk$.
\end{proposition}

\begin{proof}
The ``furthermore'' statement follows immediately from the previous statements since one may use the $\GL_m$-action on $G^{(n)}(A_m)$ induced by the action on $A_m$. In other words, there exists $\alpha\in\GL_m$ such that $\alpha(x_{1,j})=g$. One may then apply the previous statements to the map $\alpha^{-1}\varphi\alpha$ which is also $\kk[x_{i,j}]$-linear of degree $-1$, thereby showing
\[
\varphi(g^2)=\alpha(\alpha^{-1}\varphi\alpha(x_{1,j}^2))=\alpha(\lambda x_{1,j}) = \lambda g.
\]

  Next, we show that $\varphi(x_{i,j}^2)$ is a scalar
  multiple of $x_{i, j}$. By symmetry, it is enough to do so when
  $i=1$. By Lemma \ref{l:degneg1-obvious-syzygy-generalcase}, we
  may write
  \[
  \varphi(x_{1,j}^2)=\sum_i b_{i,j}x_{i,j}
  \]
  with $b_{i,j}\in \kk$.

  For any $\sigma\in S_n$, by writing
  \[
  e_1(x_1) = \sum_{j\textrm{\ odd}}(x_{1,\sigma(j)}+x_{1,\sigma(j+1)})
  \]
  and multiplying through by $\prod_{k\textrm{\ odd}}
  (x_{1,\sigma(k)}-x_{1,\sigma(k+1)})$, we arrive at the following
  syzygy:
  \begin{align}\label{first-interesting-syzygy-when-n-is-even}
    \begin{split}
      \prod_{k\textrm{\ odd}} (x_{1,\sigma(k)}-x_{1,\sigma(k+1)})
      &\varphi(e_1(x_1))=
      \\ \sum_{j\textrm{\ odd}}\prod_{\substack{k\textrm{\ odd}\\k\neq
          j}}
      &(x_{1,\sigma(k)}-x_{1,\sigma(k+1)})\bigl(\varphi(x_{1,\sigma(j)}^2)
      - \varphi(x_{1,\sigma(j+1)}^2)\bigr).
    \end{split}
  \end{align}
  We next substitute $\sum_i b_{i,j}x_{i,j}$ in place of
  $\varphi(x_{1,j}^2)$ and compare the terms on both sides of the
  syzygy which are of multidegrees
  $(\frac{n}{2}-1,0,\dots,0,1,0,\dots,0)$, where the $1$ is in the
  $i$th place and $i\neq1$. This yields
  \begin{equation}\label{eqn:first-interesting-syzygy-first-eqn}
    \sum_{j\textrm{\ odd}}\prod_{\substack{k\textrm{\ odd}\\k\neq j}}
    (x_{1,\sigma(k)}-x_{1,\sigma(k+1)})(b_{i,\sigma(j)}x_{i,\sigma(j)}
    - b_{i,\sigma(j+1)}x_{i,\sigma(j+1)}) = 0.
  \end{equation}
  We must show $b_{i,j}=0$ for all $i\neq1$. By symmetry in $i$, it is
  enough to prove $b_{2,j}=0$ for all $j$. Letting
  \[
  J_\sigma = (x_{1,\sigma(n-1)},x_{1,\sigma(n)}) +
  (x_{2,\sigma(j)}\mid j\leq n-2)+(x_{i,j}\mid i\neq1,2),
  \]
  we see $G^{(n)}(A_m)/J_\sigma$ is isomorphic to the ring $R$ in the
  statement of Lemma \ref{l:degneg1-e1-syzygy-general-case--helper},
  where $M=\frac{n}{2}-1$ and the isomorphism is given by
  $x_{1,\sigma(i)}\mapsto w_i$, $x_{2,\sigma(n-1)}\mapsto z$, and
  $x_{2,\sigma(n)}\mapsto -z$. Thus, in $R$, equation
  \eqref{eqn:first-interesting-syzygy-first-eqn} becomes
  \[
  (b_{2,\sigma(n-1)}+b_{2,\sigma(n)})z\prod_{i=1}^M(w_{2i-1}-w_{2i})=0.
  \]
  By Lemma \ref{l:degneg1-e1-syzygy-general-case--helper}, we see
  $b_{2,\sigma(n-1)}=-b_{2,\sigma(n)}$ for all $\sigma$. So, for all
  $j$, by choosing $j,k,\ell$ distinct, we see that
  $b_{2,j}=-b_{2,k}$, $b_{2,k}=-b_{2,\ell}$, $b_{2,j}=-b_{2,\ell}$
  which implies $b_{2,j}=0$.

  To finish the proof, we must calculate
  $\varphi(e_1(x_i))$. Substituting
  $\varphi(x_{1,j}^2)=a_{1,j}x_{1,j}$ into
  \eqref{first-interesting-syzygy-when-n-is-even} with $\sigma$ equal
  to the
  identity permutation, and considering terms of multidegree
  $(\frac{n}{2},0,\dots,0)$, we see
  \begin{equation}\label{eqn:multidegree-n-over-2}
    \prod_{k\textrm{\ odd}} (x_{1,k}-x_{1,k+1})\varphi(e_1(x_1)) =
    \sum_{j\textrm{\ odd}}\prod_{\substack{k\textrm{\ odd}\\k\neq j}}
    (x_{1,k}-x_{1,k+1})(a_{1,j}x_{1,j} - a_{1,j+1}x_{1,j+1}).
  \end{equation}
  Let $J=(x_{i,j}\mid i\neq 1)$. Since $G^{(n)}(A_m)/J$ is isomorphic to the
  ring $R$ from Lemma
  \ref{l:degneg1-e1-syzygy-general-case--helper2-n-even}, it follows
  that $\varphi(e_1(x_1))=\frac{1}{2}\sum_j a_{1,j}$.
\end{proof}

The following result is the last ingredient we need to show $G^{(n)}(A_m)$ has trivial negative tangents when $n$ is even.

\begin{proposition}\label{prop:degneg1-second-syzygy-generaln}
  Let $n\geq2$ and $\varphi\colon I\to G^{(n)}(A_m)$ have degree $-1$. Suppose that for all linear forms $g=\sum_i \lambda_ix_{i,j}$ with $\lambda_i\in\kk$, we have $\varphi(g^2)=\lambda g$ for some $\lambda\in\kk$. If $\varphi(x_{i,j}^2)=a_{i,j}x_{i,j}$ with $a_{i,j}\in\kk$, then for all distinct $i,k$, we have
  \[
  \varphi(x_{i,j}x_{k,j})=\frac{1}{2}a_{k,j}x_{i,j} +
  \frac{1}{2}a_{i,j}x_{k,j}.
  \]
In particular, this holds when $n\geq4$ and $m\geq2$.
\end{proposition}
\begin{proof}
Proposition \ref{prop:degneg1-e1-syzygy-general-case-n-even} shows that the hypotheses hold when $n\geq4$ and $m\geq2$. Let
\[
\varphi((x_{i,j}+x_{k,j})^2)=b^{i,k}_j (x_{i,j}+x_{k,j})\quad\textrm{and}\quad
\varphi((x_{i,j}-x_{k,j})^2)=c^{i,k}_j (x_{i,j}-x_{k,j})
\]
for $b^{i,k}_j,b^{i,k}_k\in\kk$. Suming our two expressions for $(x_{i,j}\pm x_{k,j})^2$, we see
\begin{align*}
b^{i,k}_j (x_{i,j}+x_{k,j}) + c^{i,k}_j (x_{i,j}-x_{k,j}) &= \varphi((x_{i,j}+x_{k,j})^2) + \varphi((x_{i,j}-x_{k,j})^2)\\
 &= 2\varphi(x_{i,j}^2+x_{k,j}^2) = 2(a_{i,j}x_{i,j}+a_{k,j}x_{k,j}).
\end{align*}
By linear independence of $x_{i,j}$ and $x_{k,j}$, we then have
\[
2b^{i,k}_j = a_{i,j}+a_{k,j}\quad\textrm{and}\quad 2c^{i,k}_j = a_{i,j}-a_{k,j}.
\]

Next, we see
\begin{align*}
2\varphi(x_{i,j}x_{k,j}) &=\varphi((x_{i,j}+x_{k,j})^2) - \varphi(x_{i,j}^2)-\varphi(x_{k,j}^2)\\
&=b^{i,k}_j(x_{i,j}+x_{k,j}) - a_{i,j}x_{i,j} - a_{k,j}x_{k,j}\\
&=a_{k,j}x_{i,j} + a_{i,j}x_{k,j}
\end{align*}
which finishes the proof.
\end{proof}

We may now show triviality of negative tangents for $m\geq2$ and $n\geq4$ even.

\begin{proposition}\label{prop:main-n-even-geq4}
If $m\geq2$ and $n\geq4$ is even, then $G^{(n)}(A_m)$ has trivial negative tangents.
\end{proposition}
\begin{proof}
We write $G^{(n)}(A_m)$ as the quotient of the ring $\mathscr{R}=\kk[x_{i,j}\mid
    1\leq i\leq m, 1\leq j\leq n]$ by the ideal $I$. Let $\varphi\colon I\to G^{(n)}(A_m)$ be a negatively graded
  $\mathscr{R}$-module morphism. If $\deg(\varphi)\leq-2$, then Lemma
  \ref{prop:no-degleqneg2-generalcase} shows $\varphi=0$. It remains
  to prove that if $\deg(\varphi)=-1$, then $\varphi$ is in the
  $\kk$-span of the maps $\partial_{i,j}$ given by differentiation by
  $x_{i,j}$. Proposition \ref{prop:degneg1-e1-syzygy-general-case-n-even} shows
  $\varphi(x_{i,j}^2)=a_{i,j}x_{i,j}$ for $a_{i,j}\in\kk$. Let
  $\partial:=\frac{1}{2}\sum_{i,j}a_{i,j}\partial_{i,j}$. Then
  $\varphi(x_{i,j}^2)=\partial(x_{i,j}^2)$ and Proposition
  \ref{prop:degneg1-second-syzygy-generaln} shows that
  $\varphi(x_{i,j}x_{k,j})=\partial(x_{i,j}x_{k,j})$, for distinct $i$
  and $k$. Proposition
  \ref{prop:degneg1-e1-syzygy-general-case-n-even} also shows
  $\varphi(e_1(x_i))=\partial(e_1(x_i))$, and hence
  $\varphi=\partial$.
\end{proof}

\subsection{The case where \texorpdfstring{$m\geq2$}{mge2} and
  \texorpdfstring{$n\geq5$}{nOdd} is odd} 
\label{subsec:n-odd-first-syz}

In this section, we show that if $n$ is odd, then triviality of negative tangents for $G^{(n)}(A_m)$ follows from that of $G^{(n-1)}(A_m)$. We first require a preliminary lemma.

\begin{lemma}\label{l:degneg1-e1-syzygy-general-case--helper-n-odd}
  Let $n\geq3$ be odd and
  \[
  R=\kk[w_1,\dots,w_n,z_{n-2},z_{n-1},z_n]/J,
  \]
  where
  \begin{align*}
    J =(&w_1^2,\dotsc,w_{n}^2,\, z_{n-2}^2,z_{n-1}^2,z_{n}^2,\,
    w_{n-2}+w_{n-1}+w_n,\, z_{n-2}+z_{n-1}+z_n)\\ &+ (w_iz_i\mid i\geq
    n-2) +(w_{2i-1}+w_{2i}\mid 1\leq i\leq \frac{n-3}{2}).
  \end{align*}
  If 
  \[
  f := \prod_{\substack{k\mathrm{\ odd}\\k<n}}(w_k-w_{k+1})bz_n = 0,
  \]
  then $b=0$.
\end{lemma}
\begin{proof}
  We may eliminate all variables with even subscripts as follows:
  $w_{n-1}=-(w_{n-2}+w_n)$, $z_{n-1}=-(z_{n-2}+z_n)$, and
  $w_{2i}=-w_{2i-1}$
  for $1\leq i\leq \frac{n-3}{2}$.
  Then $w_{n-1}^2=0$ tells us $w_{n-2}w_n=0$ and similarly for the
  $z$-variables. Since $w_{n-1}z_{n-1}=0$, we have
  $w_{n-2}z_n=-z_{n-2}w_n$. Thus, we see $R\simeq R'$, where
  \[
  R'=\kk[w_1,w_3,\dots,w_n,z_{n-2},z_n]/J'
  \]
  and
  \[
  J'=(w_1^2,w_3^2,\dotsc,w_n^2,\,z_{n-2}^2,z_n^2,\,w_{n-2}z_{n-2},\,
  w_nz_n,\, w_{n-2}w_n,\,z_{n-2}z_n,\,w_{n-2}z_n+z_{n-2}w_n).
  \]
  Notice that $R'$ has a $\kk$-basis given by $w_1^{c_1}w_3^{c_3}\dots
  w_{n-4}^{c_{n-4}}h$ with each $c_i\in\{0,1\}$ and $h$ an element of
  $\{1,w_{n-2},z_{n-2},w_n,z_n,w_{n-2}z_n\}$.
  In $R'$, we have
\[
    f = 2^{\frac{n-1}{2}}bw_1w_3\dotsm w_{n-2}z_n.
\]
  As a result, if $f=0$, then $b=0$.
\end{proof}

\begin{proposition}\label{prop:GnTNT-evenImpliesOdd}
If $m\geq2$ and $n\geq5$ is odd, then $G^{(n)}(A_m)$ has trivial negative tangents.
\end{proposition}
\begin{proof}
Let $\mathscr{R}=\kk[x_{i,j}]$. By Proposition \ref{prop:no-degleqneg2-generalcase}, we need only consider maps $\varphi\colon I\to G^{(n)}(A_m)=\mathscr{R}/I$ of degree $-1$. By Lemma \ref{l:degneg1-obvious-syzygy-generalcase}, we know $\varphi$ sends the ideal $K:=(x_{1,n},\dots,x_{m,n})$ to $K/I$. Noting that
\[
G^{(n-1)}(A_m)\simeq \mathcal{R}/(I+K),
\]
we obtain a commutative diagram
\[
\xymatrix{
I\ar[r]^-{\varphi}\ar@{^{(}->}[d] & G^{(n)}(A_m)\ar@{->>}[d]\\
I+K\ar[r]^-{\psi} & G^{(n-1)}(A_m)
}
\]
where we let $\psi(K)=0$. Since $n-1$ is even and at least $4$, Proposition \ref{prop:main-n-even-geq4} tells us that there are $c_{i,j}\in\kk$ such that
\[
\psi=\sum_{i=1}^m\sum_{j=1}^n c_{i,j}\partial_{i,j}=:\partial,
\]
where $\partial_{i,j}$ denotes the derivative with respect to $x_{i,j}$. We claim that $\varphi=\partial$. Replacing $\varphi$ by $\varphi-\partial$, we may therefore assume $\varphi$ has degree $-1$ and $\varphi(I) \subset K/I$.

We must show $\varphi=0$. For $j\neq n$, we see by Lemma \ref{l:degneg1-obvious-syzygy-generalcase} that $\varphi(x_{i,j}x_{\ell,j})$ is contained in $((x_{1,j},\dots,x_{m,j})\cap K)/I=0$. Similarly, $\varphi(e_1(x_i))=0$ for all $i$. It therefore remains to prove that $\varphi(x_{i,n}x_{\ell,n})=0$ for all $i,\ell$.

For this, consider the following syzygy analogous to \eqref{first-interesting-syzygy-when-n-is-even}:
  \begin{align}\label{eqn:first-interesting-syzygy-odd-case-case2}
    \begin{split}
      x_{\ell,n}\prod_{\substack{k\textrm{\ odd}\\k<n}}
      &(x_{i,k}-x_{i,k+1})\varphi(e_1(x_i))
      =\\ &x_{\ell,n}\sum_{\substack{j\textrm{\ odd}\\j<n}}\prod_{\substack{k\textrm{\ odd}\\k\neq
          j}} (x_{i,k}-x_{i,k+1})\bigl(\varphi(x_{i,j}^2) -
      \varphi(x_{i,j+1}^2)\bigr) +
      \prod_{\substack{k\textrm{\ odd}\\k<n}}(x_{i,k}-x_{i,k+1})\varphi(x_{i,n}x_{\ell,n}).
    \end{split}
  \end{align}
In our case, \eqref{eqn:first-interesting-syzygy-odd-case-case2} is particularly simple, and tells us
\[
\prod_{\substack{k \textrm{\ odd}\\k < n}} (x_{i,k} - x_{i,k+1})  \varphi(x_{i,n}x_{\ell,n}) = 0.
\]
By Lemma \ref{l:degneg1-obvious-syzygy-generalcase}, we may write $\varphi(x_{i,n}x_{\ell,n})=\sum_p b^{i,\ell}_p x_{p,n}$. Then taking into account multidegrees, the above equation tells us that for each $p$, 
\begin{equation}\label{eqn:first-interesting-syzygy-odd-case-case2-last}
\prod_{\substack{k \textrm{\ odd}\\k < n}} (x_{i,k} - x_{i,k+1}) b^{i,\ell}_p x_{p,n} = 0,
\end{equation}
and so $b^{i,\ell}_p$ vanishes by Lemma \ref{l:degneg1-e1-syzygy-general-case--helper-n-odd}. 
\end{proof}

\subsection{The case where \texorpdfstring{$n=3$}{n equals 3}}\label{sec:nequals3-new-pf}

Our goal in this subsection is to prove the following:

\begin{theorem}
  \label{thm:n3}
  When $m\ge 3$, the ring $G^{(3)}(A_m)$ has trivial negative
  tangents.
  
  Moreover, $G^{(3)}(A_m)$ lives on an irreducible locus
$Z_m:=\bbA^{2m}\times\Gr(\binom{m}{2},m^2+m+1)\subset \hilb^{d(3,m)}(\bbA^{2m})$ whose dimension is larger than that of the main component of $\hilb^{d(3,m)}(\bbA^{2m})$ for $m\geq11$.
  \end{theorem}

\begin{remark}
  \label{rmk:m2n3}
  The ideal defining $G^{(3)}(A_2)$ is given by
  \[
  I' = (z_{1}, z_{2})^2 + (w_{1}, w_{2})^2 + (z_{1}w_{2} + z_{2}w_{1})
  + (z_{1}w_{1}, z_2w_2).
  \]
  It is proved in \cite[Theorem 1.3]{Satriano--Staal--2023} that the
  ideal $J = (z_{1}, z_{2})^2 + (w_{1}, w_{2})^2 + (z_{1}w_{2} +
  z_{2}w_{1})$ is the first member in an infinite family of smooth
  points on elementary components of $\hilb^{pts}(\bbA^4)$, and that
  for most of these ideals, taking quotients by sufficiently general
  socle elements produces further examples of smooth points on
  elementary components \cite[Theorem 1.5]{Satriano--Staal--2023}.
  Here $I'$ is obtained from $J$ by adding socle elements, however
  $I'$ cannot have trivial negative tangents, because its colength is
  $6$ (see \cite{CEVV--2009}).
\end{remark}

Throughout this section, we use that $G^{(3)}(A_m)$ has an explicit basis given by $1$, the $x_{i,j}$ for $j=1,2$ and all $x_{i,1}x_{k,2}$ for $i<k$. This follows from \eqref{eqn:G3Am}.

\begin{proposition}\label{prop:degneg1-e1-syzygy-general-case-n-equals3}
If $\varphi\colon I\to G^{(3)}(A_m)$ has degree $-1$, then for any fixed $j$ and any linear form $g=\sum_i \lambda_ix_{i,j}$ with $\lambda_i\in\kk$, we have
\[
\varphi(g^2)=\lambda g
\]
for some $\lambda\in\kk$.
\end{proposition}
\begin{proof}
As in Proposition \ref{prop:degneg1-e1-syzygy-general-case-n-even}, it suffices to prove the result when $g=x_{1,1}$. Note that
\[
2x_{1,1}x_{2,1}=(x_{1,1}+x_{1,2}-x_{1,3})e_1(x_1) - x_{1,1}^2 - x_{1,2}^2 + x_{1,3}^2\in I.
\]
Letting $1 < k \leq m$, and considering the multigraded pieces of the syzygy
\begin{equation}\label{eqn:x11x12x11xk1}
x_{k,1}\varphi(x_{1,1}x_{1,2})=x_{1,2}\varphi(x_{1,1}x_{k,1})
\end{equation}
and using that $\varphi(x_{1,1}x_{k,1})$ only has $x_{i,1}$ terms, we see that $\varphi(x_{1,1}x_{1,2})$ is in the span of $x_{1,2}$, $x_{k,2}$, and the $x_{i,1}$. Since $k$ is arbitrary and $m\geq3$, this implies that $\varphi(x_{1,1}x_{1,2})$ is in the span of $x_{1,2}$ and the $x_{i,1}$.

We now consider the syzygy
\[
x_{1,2}\varphi(x_{1,1}^2)=x_{1,1}\varphi(x_{1,1}x_{1,2}).
\]
Using that $\varphi(x_{1,1}^2)$ is in the span of the $x_{i,1}$, we find $\varphi(x_{1,1}^2)$ is in the span of $x_{1,1}$.
\end{proof}

\begin{proof}[Proof of Theorem \ref{thm:n3}]
Proposition \ref{prop:no-degleqneg2-generalcase} tells us we may assume $\deg(\varphi)=-1$. Proposition \ref{prop:degneg1-e1-syzygy-general-case-n-equals3} shows $\varphi(x_{i,j}^2)=a_{i,j}x_{i,j}$ for $a_{i,j}\in\kk$, and also tells us that the hypotheses of Proposition \ref{prop:degneg1-second-syzygy-generaln} are valid. As a result, 
  \[
  \varphi(x_{i,j}x_{k,j})=\frac{1}{2}a_{k,j}x_{i,j} +
  \frac{1}{2}a_{i,j}x_{k,j}.
  \]
Thus, $\varphi=\frac{1}{2}\sum_{i,j}a_{i,j}\partial_{i,j}$ provided we can show $\varphi(e_1(x_i))=\frac{1}{2}\sum_j a_{i,j}$. By symmetry, it suffices to do so when $i=1$.

Returning to \eqref{eqn:x11x12x11xk1}, we see that if we express $\varphi(x_{1,1}x_{1,2})$ as a linear combination of the $x_{i,1}$ and $x_{i,2}$, then the $x_{1,2}$-coefficient of $\varphi(x_{1,1}x_{1,2})$ is equal to the $x_{k,1}$-coefficient of $\varphi(x_{1,1}x_{k,1})$, namely $\frac{1}{2}a_{1,1}$. We next consider the syzygy
\begin{align*}
2(x_{1,1}+x_{1,2})\varphi(e_1(x_1)) &= (x_{1,1}+x_{1,2}-x_{1,3})\varphi(e_1(x_1))\\
& = \varphi(x_{1,1}^2)+\varphi(x_{1,2}^2)-\varphi(x_{1,3}^2)+2\varphi(x_{1,1}x_{1,2})\\
& = (a_{1,1}+a_{1,3})x_{1,1} + (a_{1,2}+a_{1,3})x_{1,2} + 2\varphi(x_{1,1}x_{1,2}).
\end{align*}
Considering the $x_{1,2}$-coefficient, we find
\[
\varphi(e_1(x_1)) = \frac{1}{2}\sum_j a_{1,j}
\]
and hence $G^{(3)}(A_m)$ has trivial negative tangents.

Lastly, let $\mathscr{R}':=\kk[x_{i,j}\mid 1\leq i\leq m,1\leq j\leq 2]$ and let $M$ be the ideal generated by the $x_{i,j}x_{k,\ell}$ for $i=k$ or $j=\ell$. Then $\mathscr{R}'/M$ has basis given by $1$, all $x_{i,j}$, and all $x_{i,1}x_{k,2}$ for $i\neq k$. Hence, $D:=\dim\mathscr{R}'/M=m^2+m+1$. From \eqref{eqn:G3Am}, we see $G^{(3)}(A_m)$ is obtained from $\mathscr{R}'/M$ by additionally modding out by $\binom{m}{2}$ socle elements. Letting such socle elements vary over all possible choices yields a Grassmannian. Such points of this Grassmannian are all supported at the origin of $\bbA^{2m}$, so after allowing for translation, we obtain the locus $Z_m$. Computing, we find $\dim Z_m=(D-\binom{m}{2})\binom{m}{2} +2m$ is strictly greater than $2m(D+\binom{m}{2})$ when $m\geq11$.
\end{proof}

\subsection{Proof of Theorems \ref{thm:main} and \ref{thm:main-higher-gc}}
\label{subsec:pf-of-thm-main}

We prove Theorem \ref{thm:main-higher-gc} which implies Theorem \ref{thm:main}.

\begin{proof}[{Proof of Theorem \ref{thm:main-higher-gc}}]
By \cite[Theorem 1.2]{Jelisiejew--2019}, it suffices to prove $G:=G^{(n)}(A_m)$
  has trivial negative tangents. Note that in Theorems \ref{thm:main} and \ref{thm:main-higher-gc}, we are viewing $G$ as
  a point of $\hilb^{d(n,m)}(\bbA^{m(n-1)})$ by eliminating the
  variables $x_{i,n}$ for all $i$, as explained in the
  introduction. The elimination of the $x_{i,n}$ variables does not
  change the $\ZZ$-grading (or the $\ZZ^m$-grading) on
  $G$. Since, $T^1(G/\kk,G)_{<0}$ is independent of the
  choice of graded quotient presentation (see \S\ref{sub:tcc}), to
  prove triviality of negative tangents, it suffices to express
  $G$ as the quotient of the ring $\mathscr{R}=\kk[x_{i,j}\mid
    1\leq i\leq m, 1\leq j\leq n]$ by the ideal $I$.

For $n\geq4$ and $m\geq2$, we see $G$ has trivial negative tangents by Propositions \ref{prop:main-n-even-geq4} and \ref{prop:GnTNT-evenImpliesOdd}. We know $G$ has trivial negative tangents for $n=3$ and $m\geq3$ by Theorem \ref{thm:n3}.

It remains to consider the cases where $n\leq2$ or $m=1$ or $(n,m)=(3,2)$. When $n=1$, we see $e_1(x_i)=x_{i,1}$ and so $G^{(1)}(A_m)=\kk$, which has trivial negative tangents. When $n=2$, by eliminating all $x_{i,2}$ variables, we find $G^{(2)}(A_m)=\kk[x_{i,1}]/(x_{1,1},\dots,x_{m,1})^2\simeq A_m$. When $m\geq2$, we see $A_m$ is smoothable of dimension strictly greater than $1$, hence does not have trivial negative tangents. When $m=1$, $A_1\simeq\kk$ has trivial negative tangents.

For $m=1$ and $n\geq3$, consider the map $\varphi\colon I\to G$ given by $\varphi(x_{1,1})=x_{1,2}$ and $\varphi(g)=0$ for all minimal generators of $I$ with $g\neq x_{1,1}$. Since $x_{1,1}$ and $x_{1,2}$ are linearly independent for $n\geq3$, we see $\varphi$ is not in the $\kk$-span of the derivatives $\partial_{i,j}$. (Note that for $n=2$, $x_{1,1}$ and $x_{1,2}$ are linearly dependent.)

The remaining case $G^{(3)}(A_2)$ is checked by computer.
\end{proof}

\section{Socle Elements of the Galois closure}
\label{sec:socles-gc}

In this section, we give an explicit construction of the graded pieces
of minimal degree in the socle of the Galois closure. This makes key use of our structure result Theorem \ref{thm:structure-thm-higher-gc}. 
The main result of this section is:

\begin{theorem} \label{thm:socles-gc}
  Let $m+1\geq n\geq4$ and $D=\lceil\frac{n}{2}\rceil$. The socle elements of
  minimal degree in $G^{(n)}(A_m)$ occur in degree $D$.

  Specifically, suppose 
  $d=(d_1,\dots,d_m)$ has $\sum_i d_i=D$. Let $\mu=(D,D)$ if $n$
  is even and $\mu=(D-1,D-1,1)$ if $n$ is odd. Then every copy of
  $V_\mu\subset G^{(n)}(A_m)_d$
  is contained in the socle.
\end{theorem}


The next result reduces Theorem \ref{thm:socles-gc} to the case of the Galois closure from \cite{bhargava-satriano}.

\begin{proposition}\label{prop:socles-highgc-->gc}
Theorem \ref{thm:socles-gc} holds provided it is true when $m=n-1$.
\end{proposition}
\begin{proof}
Let $m\geq n-1$. Let $d$ be as in the statement of the theorem and fix a copy $V_\mu\subset G^{(n)}(A_m)_d$. We show $V_\mu$ is contained in the socle. Multiplication by $x_{i,j}$ maps $G^{(n)}(A_m)_d$ to a multigraded piece $G^{(n)}(A_m)_{d'}$ where $\#|\{i\mid d'_i\neq0\}|\leq\lceil n/2\rceil+1\leq n-1$ since $n\geq4$. So, after permuting the variables we can assume $d'=(d'_1,\dots,d'_{n-1},0,\dots,0)$ which we can view as living in $\NN^{n-1}$; we may therefore also view $d$ as living in $\NN^{n-1}$. Then Lemma \ref{l:multidegleqnparts} tells us we have a commutative diagram
\[
\xymatrix{
G^{(n)}(A_m)_d\ar[d]^-{\pi_d}_-{\simeq}\ar[r]^-{\cdot x_{i,j}} & G^{(n)}(A_m)_{d'}\ar[d]_-{\pi_{d'}}^-{\simeq}\\
G^{(n)}(A_{n-1})_d\ar[r]^-{\cdot x_{i,j}} & G^{(n)}(A_{n-1})_{d'}
}
\]
with vertical maps being isomorphisms of $S_n$-representations. As a result, by the $m=n-1$ case, we see
\[
\pi_{d'}(x_{i,j}(V_\mu))=x_{i,j}(\pi_d(V_\mu))=0
\]
and so $x_{i,j}(V_\mu)=0$.

Let us now show that if $\sum_i d_i<D$ and $0\neq f\in G^{(n)}(A_m)_d$, then $f$ is not a socle. As in the previous paragraph, we may view $d$ as living in $\NN^{n-1}$ and so we have an isomorphism $\pi_d\colon G^{(n)}(A_m)_d\xrightarrow{\simeq} G^{(n)}(A_{n-1})_d$. By the $m=n-1$ case, we know $\pi_d(f)$ is not a socle, so there exists $i,j$ with $x_{i,j}f\neq0$. Since $i\leq n-1\leq m$, we may multiply by $x_{i,j}$ in $G^{(n)}(A_m)$ to obtain a commutative diagram as above. It follows that $\pi_{d'}(x_{i,j}f)=x_{i,j}\pi_d(f)\neq0$ and hence $x_{i,j}f\neq0$.
\end{proof}

For the remainder of this section, we assume $m=n-1$ and freely use the notation and
terminology from \cite[Section 12.2]{bhargava-satriano} and
\cite[Chapter 2]{sagan}. We write $G(A)$ in place of $G^{(n)}(A_{n-1})$ and write partitions of $n$ as
$\mu=(\mu_1,\dots,\mu_\ell)$ with $\mu_1\geq\mu_2\geq \dotsb
\geq\mu_\ell>0$ and $\sum_i\mu_i=n$. To each such $\mu$, there is an
associated $S_n$-representation $M_\mu$ with basis given by tabloids
of shape $\mu$, see e.g., \cite[Definition 2.1.5]{sagan}. We conflate
$\mu$ with its Ferrers diagram, see \cite[Definition 2.1.1]{sagan}. We
label the boxes of $\mu$ from left-to-right and top-down, e.g.~the
$\mu_1$th box is the last box on the first row of $\mu$, and the
$(\mu_1+1)$th box is the first box on the second row of $\mu$. Let
$\{t_\mu\}$ denote the standard generator of $M_\mu$, i.e.~it is the
row-equivalence class of the tableau $t_{\mu}$ of shape $\mu$ whose
$i$th box has label $i$, i.e.\ $t_{\mu}(i)=i$.

We recall how the copies of $V_\mu\subset G(A)_d$ are constructed. Let
$\lambda$ be the partition associated to $d$ and let $T$ be a
(generalized) tableau of shape $\mu$ and content $\lambda$; this
necessarily implies $\mu\vartriangleright\lambda$. Then there is a
morphism $\theta_T\colon M_\mu\to M_\lambda$, see \cite[Definition
  2.9.3]{sagan}. As shown in the proof of \cite[Proposition
  33]{bhargava-satriano} and the surrounding discussion,
$G(A)_d=M_a/I_a\cong M_\lambda/I_\lambda$, where $a = (n-\sum_i d_i,
d_1,\dotsc,d_{n-1})$ is the \emph{ordered partition associated to $d$}
and the isomorphism is induced by a permutation $\sigma\in S_n$
satisfying
$\sigma(a)=(a_{\sigma^{-1}(1)},a_{\sigma^{-1}(2)},\dotsc,a_{\sigma^{-1}(n)})
=\lambda$.
So we obtain a map
\[
\theta_{T,\sigma}\colon M_\mu\stackrel{\theta_T}{\longrightarrow}
M_\lambda\stackrel{\sigma^{-1}}{\longrightarrow} M_a \longrightarrow
G(A)_d.
\]

There is a canonical copy of $V_\mu\subset M_\mu$. If $T$ is
semi-standard and $\mu_1=\lambda_1$, then the image of $V_\mu$ under
$\theta_{T,\sigma}$ yields a copy $V_\mu\subset G(A)_d$, see
\cite[Theorem 2.10.1]{sagan} and the proof of \cite[Proposition
  35]{bhargava-satriano}. The content of \cite[Theorem
  27]{bhargava-satriano} is that these are all the copies of $V_\mu$
in $G(A)_d$.

Explicitly, the map $\theta_{T,\sigma}$ can be described in the
following way (see the proof of \cite[Proposition
  35]{bhargava-satriano}). Let $\sigma^{-1}T$ be the tableau of shape
$\mu$ and content $a$ whose $i$th label is $(\sigma^{-1}T)(i) =
\sigma^{-1}(T(i))$.  To every tableau $S$ of shape $\mu$ and content
$a$, we obtain an element $\alpha(S)\in G(A)_d$ given as follows:~let
$x_0:=1$ and $S(i)$ denote the label in the $i$th box of $S$. Then
\begin{equation}\label{eqn:def-of-alpha}
  \alpha(S):=\prod_{i=1}^n x_{S(i)-1,i}\quad\textrm{and}\quad
  \theta_{T,\sigma}(\{t_\mu\})=\sum_{S\sim \sigma^{-1}T}\alpha(S),
\end{equation}
where $S\sim \sigma^{-1}T$ means that $S$ and $\sigma^{-1}T$ are row
equivalent, i.e.~they are both of shape $\mu$ and content $a$, and
they have the same labels in each row up to permutation.

\vspace{1em}

Ultimately, our goal is to determine how multiplication by $x_{i,j}$
acts on the copies of $V_\mu\subset G(A)_d$. In order to do so, we
begin by understanding how multiplication by $x_{i,j}$ relates to the
maps $\theta_{T,\sigma}$.  We make the following observations which
will aid in our study of the socle. First, if
$d=(d_1,\dotsc,d_{n-1})$, then multiplication by $x_{i,j}$ sends
$G(A)_d$ to $G(A)_{d'}$ where
$d'=(d_1,\dotsc,d_{i-1},d_i+1,d_{i+1},\dotsc,d_{n-1})$. We may assume
$n-\sum_i d'_i\geq d'_j$ for all $j$, otherwise $G(A)_{d'}=0$ by
\cite[Lemma 31]{bhargava-satriano}. Let $\lambda$ and $a$ be the
partition and ordered partition associated to $d$, and let
$\sigma(a)=\lambda$; let $k$ be the minimal index such that
$\lambda_k=d_i$.  Let $\lambda'$ and $a'$ be the partition and ordered
partition associated to $d'$, so that
\[
\lambda' = (\lambda_1-1, \lambda_2, \dotsc, \lambda_{k-1},
\lambda_k+1, \lambda_{k+1},\dotsc, \lambda_n).
\]
Let $\mu\vartriangleright\lambda$ with $\mu_1=\lambda_1$.

Let $T$ be a semi-standard tableau of shape $\mu$ and content
$\lambda$.
Since $T$ is semi-standard and $\mu_1=\lambda_1=a_1$, all boxes in the
first rows of $T$ and $\sigma^{-1}T$ are labelled $1$ and no other $1$
labels occur in either tableau. 
Let $\mu:=(\mu_1,\mu_2,\dots,\mu_\ell)$ and
$\widetilde{\mu}':=(\mu_1-1,\mu_2,\dots,\mu_\ell,1)$. Note that
$\widetilde{\mu}'$ might not be a partition shape since $\mu_1-1$
might be less than $\mu_2$. Let $\widetilde{T}'$ be obtained from $T$
by removing a box from the first row and creating a new $(\ell+1)$th
row with a single box labelled
$k$.  Let $\mu'$ be the partition obtained from $\widetilde{\mu}'$ by
rearranging its rows in decreasing order. Correspondingly, we permute
rows of $\widetilde{T}'$ to obtain a tableau $T'$ of shape
$\mu'$. Note that $T'$ has content $\lambda'$
but might not be semi-standard.
Moreover, we may assume $\sigma(a')=\lambda'$ (namely, that
$\sigma^{-1}(k)=i+1$), so that $\theta_{T',\sigma}\colon M_{\mu'} \to
G(A)_{d'}$ is determined by $\{t_{\mu'}\}\mapsto \sum_{S'\sim
  \sigma^{-1}T'}\alpha(S')$.  In fact, the tableau $\sigma^{-1}T'$ has
content $a'$ and, by construction, to each $S\sim \sigma^{-1}T$ there
is a uniquely corresponding tableau $S'\sim \sigma^{-1}T'$ whose rows
are the same as in $S$, except for the shorter row of $1$'s and the
additional $(\ell+1)$th row (the rows might have been rearranged).  In
particular, every $\alpha(S')$ is divisible by $x_{i,n}$ and so there
exists $\tau\in S_n$ acting on $G(A)_{d'}$
such that $\tau\alpha(S')=x_{i,j}\alpha(S)$, for all corresponding
$S'$ and $S$.  This shows that
$x_{i,j}\theta_{T,\sigma}(t_\mu)=\tau\theta_{T',\sigma}(t_{\mu'})$,
and more generally, that if $\pi\in S_n$ and
$x_{i,j}\theta_{T,\sigma}(\pi t_\mu)$ is non-zero, then it is of the
form $\rho_\pi\theta_{T',\sigma}(t_{\mu'})$ with $\rho_\pi\in S_n$.
We therefore obtain a commutative diagram
\begin{equation}\label{eqn:j-small-socle}
  \begin{gathered}
    \xymatrix{
      G(A)_d\ar[r]^-{\cdot x_{i,j}} & G(A)_{d'}\\
      M_\mu\ar[u]^-{\theta_{T,\sigma}}\ar[r]^-{\chi_{i,j}}
      & M_{\mu'}\ar[u]_-{\tau\theta_{T',\sigma}}.
    }
  \end{gathered}
\end{equation}
where $\chi_{i,j}(\pi t_{\mu})=0$ if $\pi^{-1}(j)>\lambda_1$, and
$\chi_{i,j}(\pi t_{\mu})=\tau^{-1}\rho_\pi t_{\mu'}$ otherwise.
Since $M_\mu$ has a basis given by the elements $\pi t_\mu$,
this yields a well-defined map $\chi_{i,j}$. Note that
$\chi_{i,j}(t_\mu)=t_{\mu'}$ and that $\chi_{i,j}$ is merely a vector
space map which is not $S_n$-equivariant (just as multiplication by
$x_{i,j}$ is not $S_n$-equivariant).

We now turn to the ``moreover'' statement in Theorem
\ref{thm:socles-gc}.

\begin{proposition}\label{prop:socles-of-minimal-deg}
  Let $n$, $D$, $d$, and $\mu$ be as in Theorem \ref{thm:socles-gc}.
  Then every copy of $V_\mu\subset G(A)_d$ is contained in the socle.
\end{proposition}
\begin{proof}
  Let $\lambda=(\lambda_1,\dots,\lambda_\ell)$ be the partition, and
  $a$ the ordered partition, associated to $d$. By definition of $d$
  and $D$, we necessarily have $\mu_1=\lambda_1=a_1$. Then every copy
  of $V_\mu$ in $G(A)_d$ arises as the image
  $\theta_{T,\sigma}(V_\mu)$ under the map $\theta_{T,\sigma}\colon
  M_\mu\to G(A)_d$ for some semi-standard tableau $T$ of shape $\mu$
  and content $\lambda$, and permutation $\sigma$ with
  $\sigma(a)=\lambda$. It is therefore enough to show that
  $\theta_{T,\sigma}(M_\mu)$ is contained in the socle.

Consider the commutative diagram
  \eqref{eqn:j-small-socle}. We may assume $n-\sum_k d'_k\geq d'_k$
  for all $k$, otherwise $G(A)_{d'}=0$. Letting $\lambda'$ be the
  partition associated to $d'$, we therefore have $\lambda' =
  (\lambda_1-1,\lambda_2,\dotsc,
  \lambda_{k-1},\lambda_k+1,\lambda_{k+1},\dotsc,\lambda_\ell)$ where
  $k$ is the least integer for which
  $\lambda_k=d_{i}$. On the other hand, $\mu'=(D,D-1,1)$ if $n$ is
  even and $\mu'=(D-1,D-2,1,1)$ if $n$ is odd; note $D-1\geq1$ if $n$
  is even and $D-2\geq1$ if $n$ is odd since $n\geq4$. Thus, we must
  have $\mu'_1=\mu_1=\lambda_1>\lambda'_1$, which implies, by
  \cite[Proposition 35]{bhargava-satriano} that the image of
  $\theta_{T',\sigma}$ is zero. We have therefore shown that
  $\theta_{T,\sigma}(M_\mu)$ lives in the socle.
\end{proof}

Let $\fmm$ denote the ideal of $\mathscr{R}$ generated by the
$x_{i,j}$. The following proposition plays a key role in the remainder
of this section.

\begin{proposition}
  \label{prop:no-socles-of-smaller-deg-improved-version}
  Let $d=(d_1,\dotsc,d_{n-1})$ and suppose $d_p =
    d_{p+1}=\dotsb=d_{n-1}=0$.
  for some $1\le p<n$. Let
$d'=(d_1,\dotsc,d_{p-1},1,0,\dotsc,0)$, let $\lambda$ be
  the partition associated to $d$ and $\lambda'$ be the partition
  associated to $d'$. Let $\nu=(\nu_1,\dotsc,\nu_\ell)$ be a partition
  with $\nu\vartriangleright\lambda$ and $\lambda_1=\nu_1>\nu_2$. Let
  $T$ be a semi-standard tableau of shape $\nu$ and content $\lambda$,
  and let $T'$ be obtained from $T$ by removing a box from the first
  row and creating a new $(\ell+1)$th row with a single box labelled
  $p+1$.

  Then $\nu'=(\nu_1-1,\nu_2,\dotsc,\nu_\ell,1)$ is a partition, $T'$
  is semi-standard of shape $\nu'$ and content $\lambda'$,
  $\theta_{T',1}|_{V_{\nu'}}$ is injective, and
  $\theta_{T',1}(V_{\nu'})\subset\fmm\cdot\theta_{T,1}(V_{\nu})$.
\end{proposition}
\begin{proof}
  Since $\nu_1>\nu_2$, it is clear that $\nu'$ is a partition. Next,
  since $\nu\vartriangleright\lambda$ and $\lambda_1=\nu_1>\nu_2$, we
  also have $\lambda_1>\lambda_2$. As a result,
  $\lambda'_1=\lambda_1-1$. It is then immediate that $T'$ is
  semi-standard of shape $\nu'$ and content $\lambda'$. Because
  $\lambda'_1=\nu'_1$, the structure result \cite[Theorem
    27]{bhargava-satriano} tells us that $\theta_{T',1}|_{V_{\nu'}}$
  is injective.

  Consider the maps $\chi_{p,j}$ constructed in
  \eqref{eqn:j-small-socle}. To finish the proof, it suffices to show
  that
  $V_{\nu'}$ is in the span of $\{\chi_{p,j}(V_\nu)\mid j\geq1\}$.

  Let $\{t\}\in M_{\nu'}$ be a standard tabloid.  Let $c_i$ denote the
  label appearing in the first box of the $i$th row of $t$ (so if
  $\{t\}=\{t_{\nu'}\}$ is the standard generator of $M_{\nu'}$, then
  $c_1=1$ and $c_i=\nu_1+\dotsb+\nu_{i-1}$, for $1<i\le \ell+1$) and
  let $t_i$ be the tableau of shape $\nu$ defined as follows: its
  labels are the same as those of $t$ except in the first and last
  column; in the first column, the labels from top to bottom are
  $c_1,\dotsc,c_{i-1},c_{i+1},\dotsc,c_{\ell+1}$ and the label in the
  unique box of the last column is given by $c_i$.  Consider the
  polytabloid $e_t\in V_{\nu'}$ associated to $t$ (see
  \cite[Definition 2.3.2, Theorem 2.5.2]{sagan}).  Recall that $e_t$
  is a signed sum of tabloids obtained from $t$ by permuting labels
  within columns.  Extending the map $\alpha$ from
  \eqref{eqn:def-of-alpha} by linearity, it makes sense to evaluate
  $\alpha$ on a polytabloid. Let $C_t$ and $C_{t_i}$ denote the
  column-stabilizers of $t$ and $t_i$, respectively.
  We obtain the equality
  \begin{align*}
    \theta_{T'}(e_t) &= \theta_{T'}\left( \sum_{\pi\in
      C_t}(\sgn\pi)\pi\{t\} \right) = \sum_{\pi\in
      C_t}(\sgn\pi)\pi\theta_{T'}(\{t\}) = \sum_{\pi\in
      C_t}(\sgn\pi)\pi \left( \sum_{S'\sim T'}S' \right) \\ &=
    \sum_{\pi\in C_t} \sum_{S'\sim T'} (\sgn\pi)\pi\boldsymbol{\cdot}
    S'
  \end{align*}
  and thus
  \[
  \theta_{T',1}(e_t) = \sum_{\pi\in C_t} \sum_{S'\sim T'}
  (\sgn\pi)\alpha(\pi\boldsymbol{\cdot} S') \in G(A)_{d'},
  \]
  where $\pi\in S_n$ acts on a generalized tableau $S'$ here by
  permuting positions, as in \cite[Proposition 2.9.2]{sagan}.  There
  is a natural bijection $C_i=\{ \pi\in C_t \mid \pi(n)=c_i \} \cong
  C_{t_i}, \pi\leftrightarrow\rho$, obtained by requiring that the
  labels of $\pi t$ and $\rho t_i$ coincide where their shapes
  intersect (in the sense of \cite[Definition 5.1.2]{sagan}).
  Breaking up this sum according to which $c_i$ appears in the $n$th
  box of $\nu'$, we obtain
  \begin{align*}
    \theta_{T',1}(e_t) &= \sum_{i=1}^{\ell+1}\sum_{\pi\in
      C_i}\sum_{S'\sim T'} (\sgn\pi)\alpha(\pi\boldsymbol{\cdot} S')
    = \sum_{i=1}^{\ell+1}\sum_{\pi\in C_i}\sum_{S'\sim T'} (\sgn\pi)
    \prod_{j=1}^n x_{S'(\pi^{-1}j)-1,j} \\ &=
    \sum_{i=1}^{\ell+1}\sum_{\pi\in C_i}\sum_{S'\sim T'} (\sgn\pi)
    x_{p,c_i}\prod_{\substack{j=1\\j\neq c_i}}^n
    x_{S'(\pi^{-1}j)-1,j} \\ &=
    \sum_{i=1}^{\ell+1}(-1)^{\ell+1-i}x_{p,c_i} \sum_{\rho\in
      C_{t_i}}\sum_{S'\sim T'} (\sgn\rho)
    \prod_{\substack{j=1\\j\neq c_i}}^n x_{S'(\pi^{-1}j)-1,j}.
  \end{align*}
  On the other hand, we have
  \begin{align*}
    \theta_{T,1}(e_{t_i}) &= \sum_{\rho\in C_{t_i}} \sum_{S\sim T}
    (\sgn\rho)\alpha(\rho\boldsymbol{\cdot} S) = \sum_{\rho\in
      C_{t_i}} \sum_{S\sim T} (\sgn\rho) \prod_{j=1}^n
    x_{S(\rho^{-1}j)-1,j} \in G(A)_d.
  \end{align*}
  Using the obvious bijection $\{S'\mid S'\sim T'\} \cong \{S\mid
  S\sim T\}$ and applying the appropriate permutation $\tau$ to
  reposition labels finally gives
  \[
  \tau\theta_{T',1}(e_t) =
  \sum_{i=1}^{\ell+1}(-1)^{\ell+1-i}x_{p,\tau(c_i)}
  \theta_{T,1}(e_{t_i}).
  \]
  Hence, this shows that $e_t$ is in the span of
  $\{\chi_{p,j}(V_\nu)\mid j\geq1\}$, 
  finishing the proof.
\end{proof}

Theorem \ref{thm:socles-gc} subsequently follows from the corollary below.

\begin{corollary}\label{cor:no-socles-of-smaller-deg}
  Let $n$ and $D$ be as in Theorem \ref{thm:socles-gc}. If
  $d=(d_1,\dots,d_{n-1})$ satisfies the inequality $\sum_k d_k<D$,
  then $G(A)_d$ contains no non-trivial socle elements.
\end{corollary}

\begin{proof}
  Since the socle is an $S_n$-subrepresentation of $G(A)$, it suffices
  to prove that for all partitions $\nu$, no copy of $V_\nu$ in
  $G(A)_d$ is contained in the socle.  Suppose to the contrary. Let
  $\lambda=(\lambda_1,\dotsc,\lambda_{p})$ be the partition associated
  to $d$; note that $p-1\le \sum_id_i < D$.  Because $G(A)_d\cong
  G(A)_{(\lambda_2,\dotsc,\lambda_n)}$ as $S_n$-representations, we
  may assume that the partitions associated to $d$ satisfy
  $\lambda=a$.  Suppose there exists some
  $\nu\vartriangleright\lambda$ with $\nu_1=\lambda_1$ and some
  semi-standard tableau $T$ of shape $\nu$ and content $\lambda$ such
  that $\theta_{T,1}(V_\nu)\subset G(A)_d$ is contained in the socle.
  We will examine multiplication by $x_{p,j}$, where $p$ is the number
  of rows of $\lambda$.

  Let $\nu=(\nu_1,\dots,\nu_\ell)$ and
  $d'=(d_1,\dotsc,d_{p-1},1,0,\dotsc,0)$.
  Since $\lambda_1>n-D$, we have $\lambda_1\geq\frac{n+1}{2}$
  regardless of whether $n$ is even or odd. This implies
  $\lambda_1>\lambda_2$; indeed, otherwise
  $\lambda_1+\lambda_2=2\lambda_1>n$, a contradiction. Therefore, the
  partition $\lambda'$ associated to $d'$ satisfies
  $\lambda'_1=\lambda_1-1$. Similarly, $\nu_1=\lambda_1>n-D$ and so
  $\nu_1>\nu_2$. We may therefore apply Proposition
  \ref{prop:no-socles-of-smaller-deg-improved-version}.  By our
  assumption that $\theta_{T,1}(V_\nu)\subset G(A)_d$ is contained in
  the socle, we see that $x_{p,j}\theta_{T,1}(V_\nu)=0$ and so
  $\theta_{T',1}(\chi_{p,j}(V_\nu))=0$. However, Proposition
  \ref{prop:no-socles-of-smaller-deg-improved-version} tells us every
  non-zero element $v\in V_{\nu'}$ is in the span of
  $\{\chi_{p,j}(V_\nu)\mid j\geq1\}$. 
  This implies that
  $\theta_{T',1}(v)=0$, contradicting that fact that
  $\theta_{T',1}|_{V_{\nu'}}$ is injective.
\end{proof}

The following corollary characterizes the annihilator of the square of
the maximal ideal $\fmm$.  This will play an important role in Theorem
\ref{thm:main-higher-gc-socle} below.

\begin{corollary}\label{cor:no-annm2-of-smaller-deg}
  Let $n$ and $D$ be as in Theorem \ref{thm:socles-gc}. Let
  $d=(d_1,\dotsc,d_{n-1})$ satisfy the inequality $\sum_k d_k<D$ and
  let $\lambda$ and $a$ be the partitions associated to $d$.  If $n$
  is even, then $G(A)_d\cap \Ann(\fmm^2)=0$ holds.  If $n$ is odd, let
  $\varepsilon = (D,D-1)$; then we have
  $G(A)_d\cap\Ann(\fmm^2)\subseteq
  \sum_T\theta_{T,\sigma}(V_{\varepsilon})$, where $T$ runs over
  semi-standard Young tableaux of shape $\varepsilon$ and content
  $\lambda$, and $\sigma(a)=\lambda$.
\end{corollary}

\begin{proof}
  Let $\lambda = (\lambda_1,\dotsc,\lambda_p)$.  As
  $G(A)_d\cap\Ann(\fmm^2)$ is a subrepresentation, it suffices to
  determine which $V_{\nu}$'s are contained in the intersection;
  again, we may assume $a=\lambda$.  Let $\nu =
  (\nu_1,\dotsc,\nu_{\ell})\vartriangleright\lambda$ with
  $\nu_1=\lambda_1$ and let $T$ be semi-standard of shape $\nu$ and
  content $\lambda$.  Because $\sum_kd_k<D$, we have $\nu_1>\nu_2+1$,
  unless $n$ is odd and $\nu=\varepsilon.$

  If $\nu_1>\nu_2+1$, then $\lambda_1>\lambda_2+1$ as well.  In this
  case, noting that $n\ge 4$ and $p-1<D$, we see
  $\lambda'=(\lambda_1-2,\lambda_2,\lambda_3,\dotsc,\lambda_p,1,1)$
  and $\nu'=(\nu_1-2,\nu_2,\nu_3,\dotsc,\nu_{\ell},1,1)$ are
  partitions.  Let $T'$ be the semi-standard tableau obtained from $T$
  by labelling the boxes on the $(\ell+1)$th and $(\ell+2)$th rows
  with $p+1$ and $p+2$, respectively.  Applying Proposition
  \ref{prop:no-socles-of-smaller-deg-improved-version} twice, we see
  that $\theta_{T',1}(V_{\nu'})\subseteq \fmm^2\theta_{T,1}(V_{\nu})$.
  Because $\theta_{T',1}|_{V_{\nu'}}$ is injective, we find that
  $\theta_{T,1}(V_{\nu})$ cannot be contained in $\Ann(\fmm^2)$.
\end{proof}

\section{Trivial negative tangents for quotients by minimal degree socles}
\label{TNT-min-socs}

In this section, we prove Theorem
\ref{thm:main-higher-gc-socle} and Corollary
\ref{cor:main-higher-gc-socle}, giving a formula for how large we may take the value of $r$. We begin with two examples showing that condition \ref{socleodd} in Theorem
\ref{thm:main-higher-gc-socle} is neither vacuous or superfluous.

\begin{example}
  \label{eg:n-equals-5}
One checks via computer that we may quotient $G(A_4)=G^{(5)}(A_4)$ by all copies of $V_\mu$ in $G(A_4)_d$ for all $d$ whose associated partition is $(1,1,1)$, i.e., in Theorem \ref{thm:main-higher-gc-socle}, we may take $n=5$, $m=4$, and $r=40$.
\end{example}

\begin{example}
  \label{eg:n-equals-7}
  Let $J$ be the ideal of $G(A_6)=G^{(7)}(A_6)$ generated by all socle
  elements in degree $D=4$.  Let $T$ be the tableau of shape
  $\varepsilon=(4,3)$ and content $\lambda=(4,1,1,1)$, so the entries
  of $T$ are $1,1,1,1,2,3,4$ read left-to-right and top-down.  Let $t$
  be the standard tableau of shape $\varepsilon$ whose $i$th
    box has label $i$.  One can check by
  computer that $\theta_{T,1}(e_t)$ yields a non-zero socle element of
  $B$ in degree $3$.
\end{example}

In order to prove that $B$ has trivial negative tangents, we compare $T^1(B/\kk,B)$ to $T^1(G^{(n)}(A_m)/\kk,B)$ and $T^1(B/G^{(n)}(A_m),B)$. The following gives the appropriate vanishing result we need for $T^1(G^{(n)}(A_m)/\kk,B)_{<0}$.

\begin{proposition}\label{prop:TNT-TAkB}
  Let $n\ge 4$ and $m\geq2$.  
  Let $s_1,\dots,s_r\in G^{(n)}(A_m)$ be socle elements and assume each $s_i$
  is contained in a multigraded component $G^{(n)}(A_m)_{d_i}$ where
  $d_i=(d_{i,1},\dots,d_{i,m})$ satisfies:
  \begin{itemize}[leftmargin=.25in]
  \item $\sum_j d_{i,j}\geq 3$,
  \item if $D_i:=\sum_j d_{i,j}\in\{\lfloor\frac{n}{2}\rfloor,\lceil\frac{n}{2}\rceil\}$, then every $d_{i,j}\leq D_i-2$.
  \end{itemize}
  If $J=(s_1,\dots,s_r)$, then $B:=G^{(n)}(A_m)/J$ satisfies
  $T^1(G^{(n)}(A_m)/\kk,B)_{<0}=0$.
\end{proposition}

\begin{proof}
  Let $\mathscr{R}=\kk[x_{i,j}]$ and let $\varphi\colon I\to B$ be a
  graded $\mathscr{R}$-module map of negative degree. 

  First suppose $\deg(\varphi)\leq-2$. Then
  \eqref{eqn:syzygy-used-to-show-deg--2-vanishes} holds modulo $J$;
  the terms in the equation have degree at most $1$ and none of the $s_i$ have degree $1$, so we see
  \eqref{eqn:syzygy-used-to-show-deg--2-vanishes} remains true. Thus,
  the proof of Proposition \ref{prop:no-degleqneg2-generalcase} shows
  $\varphi=0$.

  Now suppose $\deg(\varphi)=-1$.  Note first that \eqref{eqn:first-quadratic-syzgy} holds
  modulo $J$ and since all terms in the ensuing expansion
  have degree $2<D_i$ for all $i$, the equations following
  \eqref{eqn:first-quadratic-syzgy} also remain true. As a result, the
  conclusion of Lemma \ref{l:degneg1-obvious-syzygy-generalcase} is
  still valid. 

To handle the case where $n$ is even, it therefore suffices to show that the proof of Proposition \ref{prop:main-n-even-geq4} still applies in our setting. For this, it is enough to show that the conclusions of  Proposition \ref{prop:degneg1-e1-syzygy-general-case-n-even} and
  Proposition \ref{prop:degneg1-second-syzygy-generaln} remain  valid. Note that \eqref{first-interesting-syzygy-when-n-is-even}
  holds modulo $J$. By assumption, $J$ contains no elements of
  multidegree $(\frac{n}{2}-1,0,\dots,0,1,0,\dots,0)$, where the $1$
  is in the $i$th place and $i\neq1$. As a result,
  \eqref{eqn:first-interesting-syzygy-first-eqn} holds. Similarly, $J$
  contains no elements of multidegree $(\frac{n}{2},0,\dots,0)$, and
  so \eqref{eqn:multidegree-n-over-2} holds. Hence, the conclusion of
  Proposition \ref{prop:degneg1-e1-syzygy-general-case-n-even} remains
  valid. Finally, Proposition \ref{prop:degneg1-second-syzygy-generaln} still
  holds since the equations considered in the proof only involve syzygies of degree at most $2$.  We conclude that
  $\varphi$ is a $\kk$-linear combination of derivatives $\partial_{i,j}$.

We now turn to the case where $n$ is odd. We must show the proof of Proposition \ref{prop:GnTNT-evenImpliesOdd} applies. As shown above, the conclusions of Proposition \ref{prop:no-degleqneg2-generalcase} and Lemma \ref{l:degneg1-obvious-syzygy-generalcase} hold so we again have a commutative diagram
\[
\xymatrix{
I\ar[r]^-{\varphi}\ar@{^{(}->}[d] & G^{(n)}(A_m)/J\ar@{->>}[d]\\
I+K\ar[r]^-{\psi} & G^{(n-1)}(A_m)/J'
}
\]
where $J'$ is the ideal generated by the images of the $s_i$ in $G^{(n-1)}(A_m)$, and where we let $\psi(K)=0$. Since $n-1\geq4$ is even, there are $c_{i,j}\in\kk$ such that $\psi=\sum_{i=1}^m\sum_{j=1}^n c_{i,j}\partial_{i,j}=:\partial$. Replacing $\varphi$ by $\varphi-\partial$, we may assume $\varphi(I)\subset K + K'$, where $K'$ is generated by all $s_i$ mapping to $0$ in $G^{(n-1)}(A_m)$. We must prove $\varphi=0$. Since all $D_i\geq3$, we see $\varphi(x_{i,j}x_{\ell,j})=\varphi(e_1(x_i))=0$ for $j\neq n$. Using that \eqref{eqn:first-interesting-syzygy-odd-case-case2} holds and using our hypothesis on the $d_{i,j}$, we see \eqref{eqn:first-interesting-syzygy-odd-case-case2-last} holds and so we again find $\varphi(x_{i,n}x_{\ell,n})=0$.

  We have therefore shown that $T^1(G^{(n)}(A_m)/\kk,B)_{<0}=0$.
\end{proof}

We next give a general criterion to show $T^1(B'/A',B')_{<0}=0$, where
$A'\to B'$ is a surjection of rings. (We add superscript primes so as
not to conflict with our running notation
$G^{(n)}(A_m)=\mathscr{R}/I$, $B = G^{(n)}(A_m)/J$, etc.) In the proof of Theorem \ref{thm:main-higher-gc-socle} we apply this criterion to
show that $T^1(B/G^{(n)}(A_m),B)_{<0}=0$.

\begin{proposition}\label{prop:TNT-TBAB}
  Let $I'\subset S'=\kk[y_1,\dots,y_n]$ be a graded ideal. Let
  $A'=S'/I'$ and $D$ be the minimal degree of a socle element of
  $A'$. If $s_1,\dotsc,s_r$ are linearly independent degree $D$ socle
  elements modulo $I'$, $J'=(s_1,\dotsc,s_r)+I'$, and $B'=S'/J'$ has
  no socle elements in degrees less than $D$, then
  $T^1(B'/A',B')_{<0}=0$.
\end{proposition}
\begin{proof}
  We first construct the truncated cotangent complex
  $L_{B'/A',\bullet}$. We have a surjective map $\pi\colon A'\to B'$,
  from which we see that $L_{B'/A',0}=\Omega^1_{A'/A'}\otimes_{A'}
  B'=0$. Next, $\ker\pi=J'/I'$ which, as a graded $A'$-module, is
  isomorphic to $\kk(-D)^{\oplus r}$. We then have a short exact
  sequence
  \[
  0\to \fmm_{A'}(-D)^{\oplus r}\to A'(-D)^{\oplus r}\to J'/I'\to 0
  \]
  and we let $\Kos\subset\fmm_{A'}(-D)^{\oplus r}$ denote the Koszul
  syzygies. From this data, we see the truncated cotangent complex is
  given by
  \[
  L_{B'/A',\bullet} \colon 0 \longleftarrow {B'}(-D)^{\oplus r}
  \longleftarrow \fmm_{A'}(-D)^{\oplus r}/\Kos.
  \]
  Every element of $T^1(B'/A',B')$ is represented by a $B'$-module map
  $\varphi\colon B'(-D)^{\oplus r}\to B'$ which vanishes when
  precomposed with $\fmm_{A'}(-D)^{\oplus r}/\Kos\to B'(-D)^{\oplus
    r}$; in particular, it vanishes when precomposed with
  $\fmm_{A'}(-D)^{\oplus r}\to B'(-D)^{\oplus r}$, and hence must send
  the generators of $B'(-D)^{\oplus r}$ to $\Soc(B')$. However, there
  are no non-zero socle elements of degree less than $D$, so if
  $\varphi$ is negatively graded, it must vanish.
\end{proof}

Combining Propositions \ref{prop:TNT-TAkB} and \ref{prop:TNT-TBAB} will yield all statements of Theorem \ref{thm:main-higher-gc-socle}, except for \ref{socleeven} and \ref{socleodd}; we address \ref{socleeven} and \ref{socleodd} in the result below

\begin{proposition}
  \label{prop:no-low-deg-soc-B}
Let $m\geq n-1$, and keep the notation and hypotheses of Theorem \ref{thm:main-higher-gc-socle}. If \ref{socleeven} or \ref{socleodd} hold, then $B$ has no socle elements of degree less than $D$.
\end{proposition}

\begin{proof}
  Any such socle element of $B$ would yield a non-zero element in
  $G^{(n)}(A_m)\cap \Ann(\fmm^2)$ of degree less than $D$.  By Corollary
  \ref{cor:no-annm2-of-smaller-deg}, $n$ must be odd and all such
  elements live in copies of $V_{\varepsilon}$ in degree $D-1$.
\end{proof}

\begin{proof}[{Proof of Theorem \ref{thm:main-higher-gc-socle}}]
  The pair of natural ring maps $\kk \to G^{(n)}(A_m) \to B$ yields a long
  exact sequence containing the following portion, see e.g.,
  \cite[Theorem 3.5]{Hartshorne--2010}:
  \[
  \dotsb \longrightarrow T^1(B/G^{(n)}(A_m),B) \longrightarrow T^1(B/\kk,B)
  \longrightarrow T^1(G^{(n)}(A_m)/\kk,B) \longrightarrow \dotsb.
  \]
We know $T^1(G^{(n)}(A_m)/\kk,B)_{<0}$ vanishes by Proposition \ref{prop:TNT-TAkB} and $T^1(B/G^{(n)}(A_m),B)_{<0}$ vanishes by Proposition \ref{prop:TNT-TBAB}. Thus, $T^1(B/\kk,B)_{<0} = 0$. It follows from \cite[Theorem 1.2]{Jelisiejew--2019} that all components of the Hilbert scheme containing $[B]$ are elementary. The remaining statements of Theorem \ref{thm:main-higher-gc-socle} are handled by Proposition \ref{prop:no-low-deg-soc-B}.
\end{proof}

Lastly, we turn to Corollary \ref{cor:main-higher-gc-socle}, giving a formula for the
value of $r$ appearing in Theorem \ref{thm:main-higher-gc-socle}.

\begin{proof}[{Proof of Corollary \ref{cor:main-higher-gc-socle}}]
  Let $D$ and $\mu$ be as in Theorem \ref{thm:socles-gc}. If
  $d=(d_1,\dots,d_m)$ with $\sum_i d_i=D$, then the socle of
  $G^{(n)}(A_m)$ contains every copy of $V_\mu$ in $G^{(n)}(A_m)_d$. When $n\geq6$ is even,
  Theorem \ref{thm:main-higher-gc-socle} tells us we may
  take the quotient by all such copies $V_\mu$ and maintain triviality
  of negative tangents, provided that every
  $d_i\leq\lfloor\frac{n}{2}\rfloor-2$.
  Letting $\lambda$ be the partition associated to $d$,
  the above condition exactly means that $\lambda$ is in the set $\cP$ defined before the statement of Corollary \ref{cor:main-higher-gc-socle}.
  By Theorem \ref{thm:structure-thm-higher-gc}, all such copies of $V_\mu$
  together have dimension
  \[
  R_n:=
  \sum_{\lambda\in\cP}m_\lambda K_{\mu\lambda}\dim V_\mu.
  \]
  It therefore suffices to show that $R_n$ agrees with the function
  $R(n)$.

  Using the hook length formula \cite[Theorem 3.10.2]{sagan}, one
  computes
  \[ 
  \dim V_\mu =
  \frac{n!}{(\frac{n}{2})!(\frac{n}{2}+1)!}=C_{n/2}, 
  \]
  for even $n$.
  Since $\lambda_1=\mu_1$, it is easy to see that when $n$ is even,
  there is precisely one semi-standard Young tableau of shape $\mu$
  and content $\lambda$, hence $K_{\mu\lambda}=1$.
  This establishes that $R_n=R(n)$.
\end{proof}

\section{Smoothness and obstructions for \texorpdfstring{$G^{(n)}(A_m)$}{GnAm}}
\label{sec:vnos}

We now turn to Theorem \ref{thm:VNOS}, which addresses precisely when the obstruction space for $G^{(n)}(A_m)$ vanishes. In particular, this shows $G^{(n)}(A_m)$ is smooth for $n=3$ and for some sporadic cases where $n=4$.

\begin{proof}[{Proof of Theorem \ref{thm:VNOS}}]
We first handle the case when $n=3$. By \eqref{eqn:G3Am}, we see the ideal defining $G^{(3)}(A_m)$ is generated by quadratics. So all syzygies live in degree at least $3$. Thus, any map $\psi\colon L_{G^{(3)}(A_m),2}\to G^{(3)}(A_m)$ of non-negative degree must have image in degree at least $3$, however $G^{(3)}(A_m)$ is concentrated in degree at most $2$. Thus, $T^2(G^{(n)}(A_m)/\kk,G^{(n)}(A_m))_{\geq0}$ vanishes.

The remaining arguments proceed as follows. We choose a homogeneous element $s\in L_{2,G^{(n)}(A_m)}$ corrresponding to a minimal syzygy of $I$, and choose a homogeneous non-zero socle element $f\in G^{(n)}(A_m)$ with $\deg(f)\geq\deg(s)$. We then let
\[
\psi\colon L_{2,G^{(n)}(A_m)}\to G^{(n)}(A_m)
\]
send $s$ to $f$ and send all other minimal syzygies to $0$. The map $\psi$ is well-defined since $f$ is in the socle; note that $\deg(\psi)=\deg(f)-\deg(s)\geq0$. We then show $[\psi]\neq0$.

Consider first the case where $m\geq2$ and $n\geq8$. Let $\ell=2$ (resp.~$\ell=3$) if $n$ is even (resp.~odd), and let
\[
f = \prod_{j=1}^\ell x_{2,j}\cdot \prod_{j=\ell+1}^{(n+\ell)/2}x_{1,j}.
\]
We see that the multidegree $(\frac{n-\ell}{2},\ell,0,\dots,0)$ piece of $G^{(n)}(A_m)$ is generated as an $S_n$-representation by $f$. Since this multidegree piece is non-zero by Theorem \ref{thm:structure-thm-higher-gc}, we see $f\neq0$. Furthermore, $f$ is in the socle by Lemma \ref{l:lemma31}\ref{lemma31part}. Now, if $n$ is even, let $s$ be the syzygy among $e_1(x_2)$ and $x_{2,1}^2,\dots,x_{2,n}^2$ given by equation (8) with $\sigma=\id$. If $n$ is odd, let $s$ be the syzygy among $e_1(x_2)$, $x_{2,1}^2,\dots,x_{2,n-1}^2$, and $x_{2,n}$ given by equation \eqref{eqn:first-interesting-syzygy-odd-case-case2}. Then we consider the map $\psi$ sending $s$ to the socle $f$, and sending all other syzygies to zero; note that
\[
\deg(\psi)=\deg(s)-\deg(f)=0.
\]
If $[\psi]=0$, then $\psi$ extends to a map $\varphi\colon L_{G^{(n)}(A_m)/\kk,1} \to G^{(n)}(A_m)$. This implies that $f$ is expressible as an $\kk[x_{i,j}]$-linear combination of monomials all of which have multidegree $(d_1,\dots,d_m)$ with $d_2\geq (n+\ell)/2 - 2$, e.g., if $n$ is even, we then have
\[
f = \prod_{k\textrm{\ odd}} (x_{2,k}-x_{2,k+1})
\varphi(e_1(x_1)) - 
 \sum_{j\textrm{\ odd}}\prod_{\substack{k\textrm{\ odd}\\k\neq
          j}}
    (x_{2,k}-x_{2,k+1})\bigl(\varphi(x_{2,j}^2)
      - \varphi(x_{2,j+1}^2)\bigr).
\]
Since $n\geq8$, we see $(n+\ell)/2 - 2>\ell$, which is a contradiction. Thus, $[\psi]\neq0$.

It remains to consider the cases $4\leq n\leq 7$. Assuming $m\geq n+1$, we take
\[
f=x_{1,1}\dots x_{n-1,n-1}
\]
and $s$ to be the cubic syzygy between $x_{n,n}^2$ and $x_{n,n}x_{n+1,n}$. Again by Theorem \ref{thm:structure-thm-higher-gc}, $f$ is non-zero since it generates the mutidegree piece $(1,\dots,1,0,\dots,0)$ of $G^{(n)}(A_m)$ which is the sign representation; it is a socle by Lemma \ref{l:lemma31}\ref{lemma31part}. We see $\deg(\psi) = n-4\geq0$. If $[\psi]\neq0$, then an argument analogous to the above shows that $f\in(x_{n,n},x_{n+1,n})$ which is not the case, as one sees from multidegrees.

We have now reduced to finitely many cases, which one may check by computer. Alternatively, one may further reduce the number of computer checks by using arguments analogous to the above. For $m\geq3$ and $n\geq5$, one may take $s$ to be the cubic syzygy between $x_{2,n}^2$ and $x_{2,n}x_{3,n}$, and 
\[
f = x_{1,1}x_{1,2}\dots x_{1,\lfloor n/2\rfloor}
\]
for $n$ even and 
\[
f = x_{1,1}x_{1,2}\dots x_{1,\lfloor n/2\rfloor}x_{2,\lceil n/2\rceil}.
\]
for $n$ odd. Thus, one need only check the cases $(n,m)\in\{(4,m)\mid 2\leq m\leq 4\}\cup\{(n,2)\mid 4\leq n\leq 7\}$ by computer.
\end{proof}

\begin{remark}\label{rmk:VNOS}
Determining whether or not $G^{(n)}(A_m)$ is smooth appears to be a subtle question. Outside of specialized cases, there are few techniques to show $G^{(n)}(A_m)$ is smooth when the obstruction space is non-vanishing. Most relevant to the current paper, in \cite{Jelisiejew--2019}, Jelisiejew constructs an infinite family of algebras which have trivial negative tangents and define smooth points of the Hilbert scheme, yet their obstruction spaces do \emph{not} vanish, see \cite[\S1.1]{Satriano--Staal--2023} for a detailed overview of Jelisiejew's method of proof. Unfortunately, Jelisiejew's techniques are not applicable in our setting. Indeed, when $m\geq3$, \cite[Theorem 4.12]{Jelisiejew--2019} does not apply since the monomial ideal $M$ generated by quadratics of $I$ does not define a smooth point of the Hilbert scheme. On the other hand, when $m=2$, the ideal $M$ does define a smooth point, however \cite[Theorem 4.12]{Jelisiejew--2019} is still not applicable since the map $\partial_{\geq0}$ need not be surjective, e.g., $\partial_0$ is not surjective for $(m,n) = (2,5)$. It is especially worth noting that when $M$ is not smooth, Jelisiejew's strategy to show triviality of negative tangents also does not apply in our setting, hence the need for our explicit computations in Section \ref{sec:triviality-neg-tngts}.
\end{remark}

In light of Remark \ref{rmk:VNOS}, we pose:

\begin{question}\label{q:VNOS}
For $n\geq4$ and $(n,m)\notin\{(4,2),(4,3)\}$, does $G^{(n)}(A_m)$ define a singular point of $\hilb^{d(n,m)}(\bbA^{n(m-1)})$?
\end{question}

\section{Data for rings of rank at most \texorpdfstring{$6$}{6}}
\label{sec:last}

We conclude by gathering data on Question \ref{q:higher-gc-TNT} for small $n$ and all isomorphism classes of small rank rings. 
As mentioned in the introduction, in \cite[Table 1]{Poonen--2008--iso}, Poonen lists all isomorphism
classes of $\kk$-algebras of rank $m\leq 6$; for any $m>6$, it is well-known that there are infinitely many isomorphism classes of rings of ranks $m$. Table \ref{table} summarizes computations done in \emph{Macaulay2}; a check mark ($\checkmark$) indicates that $G^{(n)}(A)$ has trivial negative tangents, $\times$ indicates it does not, and $?$ indicates that our computation ran for at least $12$ hours without terminating. Note that $G^{(1)}(A)\simeq\kk$ always trivially has trivial negative tangents, while $G^{(2)}(A) \cong
A/(x_1^2,\dotsc,x_m^2)$ 
never has trivial negative tangents when
$2\le \rank A\le 6$, as the Hilbert scheme is irreducible in these
cases. Thus, Table \ref{table} considers the values $n\geq3$.



\begin{table}[ht] 
  \caption[Isom Classes]{A list of all isomorphism classes of
    $\kk$-algebras $A = S/I$ of low rank where $G^{(n)}(A)$ has
    trivial negative tangents
  }
  \centering
  \resizebox{.65\columnwidth}{!}{%
  \begin{tabular}{cc|c|c|c|c|}
    \cline{3-6}
    & & \multicolumn{4}{c|}{$G^{(n)}(A)$ has TNT} \\
    \cline{1-6}
    \multicolumn{1}{|c|}{$S$} & $I$ & $n=3$ & $4$ & $5$ & $6$ \\
    \hline\hline
    \multicolumn{6}{|c|}{$\rank 2$} \\
    \hline
    \multicolumn{1}{|c|}{$\kk[x]$} & $\llrr{x^2}$ & $\times$ &
    $\times$ & $\times$ & $\times$ \\
    \hline
    \multicolumn{6}{|c|}{$\rank 3$} \\
    \hline
    \multicolumn{1}{|c|}{$\kk[x]$} & $\llrr{x^3}$ & $\times$ &
    $\times$ & $\times$ & $\times$ \\
    \multicolumn{1}{|c|}{$\kk[x,y]$} & $\llrr{x,y}^2$ & $\times$ &
    \checkmark & \checkmark & \checkmark \\
    \hline
    \multicolumn{6}{|c|}{$\rank 4$} \\
    \hline
    \multicolumn{1}{|c|}{$\kk[x]$} & $\llrr{x^4}$ & $\times$ &
    $\times$ & $\times$ & $\times$ \\
    \multicolumn{1}{|c|}{$\kk[x,y]$} & $\llrr{x^2,xy,y^3}$ & $\times$
    & $\times$ & $\times$ & \checkmark \\ 
    \multicolumn{1}{|c|}{$\kk[x,y]$} & $\llrr{x^2,y^2}$ & \checkmark &
    $\times$ & \checkmark & \checkmark \\
    \multicolumn{1}{|c|}{$\kk[x,y,z]$} & $\llrr{x,y,z}^2$ & \checkmark
    & \checkmark & \checkmark & \checkmark \\
    \hline
    \multicolumn{6}{|c|}{$\rank 5$} \\
    \hline
    \multicolumn{1}{|c|}{$\kk[x]$} & $\llrr{x^5}$ & $\times$ &
    $\times$ & $\times$ & $\times$ \\
    \multicolumn{1}{|c|}{$\kk[x,y]$} & $\llrr{x^2,xy,y^4}$ & $\times$
    & $\times$ & $\times$ & $\times$ \\
    \multicolumn{1}{|c|}{$\kk[x,y]$} & $\llrr{x^2+y^3,xy}$ & $\times$
    & $\times$ & $\times$ & ? \\
    \multicolumn{1}{|c|}{$\kk[x,y]$} & $\llrr{x^3,xy,y^3}$ & $\times$
    & $\times$ & \checkmark &
    \checkmark\footnotemark[3] \\
    \multicolumn{1}{|c|}{$\kk[x,y]$} & $\llrr{x^2,xy^2,y^3}$ &
    $\times$ & $\times$ & $\times$ & \checkmark\footnotemark[3] \\
    \multicolumn{1}{|c|}{$\kk[x,y,z]$} & $\llrr{x^2,xy,xz,y^2,yz,z^3}$
    & $\times$ & $\times$ & \checkmark & ? \\
    \multicolumn{1}{|c|}{$\kk[x,y,z]$} & $\llrr{x^2,xy,xz,y^2,z^2}$ &
    $\times$ & \checkmark & \checkmark & \checkmark\footnotemark[3] \\
    \multicolumn{1}{|c|}{$\kk[x,y,z]$} &
    $\llrr{x^2+y^2,x^2+z^2,xy,xz,yz}$ & \checkmark & \checkmark &
    \checkmark\footnotemark[3] & ? \\
    \multicolumn{1}{|c|}{$\kk[x,y,z,w]$} & $\llrr{x,y,z,w}^2$ &
    \checkmark & \checkmark & \checkmark & \checkmark \\
    \hline
    \multicolumn{6}{|c|}{$\rank 6$} \\
    \hline
    \multicolumn{1}{|c|}{$\kk[x]$} & $\llrr{x^6}$ & $\times$ &
    $\times$ & $\times$ & $\times$ \\
    \multicolumn{1}{|c|}{$\kk[x,y]$} & $\llrr{x^2,xy,y^5}$ & $\times$
    & $\times$ & $\times$ & $\times$\footnotemark[3] \\
    \multicolumn{1}{|c|}{$\kk[x,y]$} & $\llrr{x^2+y^4,xy}$ & $\times$
    & $\times$ & $\times$ & ? \\ 
    \multicolumn{1}{|c|}{$\kk[x,y]$} & $\llrr{x^3,xy,y^4}$ & $\times$
    & $\times$ & $\times$ & ? \\
    \multicolumn{1}{|c|}{$\kk[x,y]$} & $\llrr{x^3+y^3,xy}$ & $\times$
    & \checkmark & $\times$ & ? \\ 
    \multicolumn{1}{|c|}{$\kk[x,y]$} & $\llrr{x^2,xy^2,y^4}$ &
    $\times$ & $\times$ & $\times$ & ? \\
    \multicolumn{1}{|c|}{$\kk[x,y]$} & $\llrr{x^2+y^3,xy^2,y^4}$ &
    $\times$ & $\times$ & ? & ? \\
    \multicolumn{1}{|c|}{$\kk[x,y]$} & $\llrr{x^2,y^3}$ & $\times$ &
    \checkmark & \checkmark & ? \\
    \multicolumn{1}{|c|}{$\kk[x,y]$} & $\llrr{x,y}^3$ & $\times$ &
    $\times$ & \checkmark\footnotemark[3] & ? \\
    \multicolumn{1}{|c|}{$\kk[x,y,z]$} & $\llrr{x^2,xy,xz,y^2,yz,z^4}$
    & $\times$ & $\times$ & $\times$\footnotemark[3] & ? \\
    \multicolumn{1}{|c|}{$\kk[x,y,z]$} &
    $\llrr{x^2,xy,xz,y^2+z^3,yz,z^4}$ & $\times$ & $\times$ &
    $\times$\footnotemark[3] & ? \\
    \multicolumn{1}{|c|}{$\kk[x,y,z]$} & $\llrr{x^2,xy+z^3,xz,y^2,yz,z^4}$
    & $\times$ & $\times$ & $\times$\footnotemark[3] & ? \\
    \multicolumn{1}{|c|}{$\kk[x,y,z]$} & $\llrr{x^3,xy,y^2-xz,yz,z^2}$
    & $\times$ &\checkmark & \checkmark\footnotemark[3] & ? \\
    \multicolumn{1}{|c|}{$\kk[x,y,z]$} & $\llrr{x^2+y^2-xz,xy,xz-yz,z^2}$
    & \checkmark & \checkmark & ? & ? \\
    \multicolumn{1}{|c|}{$\kk[x,y,z]$} & $\llrr{x^2,xy,xz,y^2,yz^2,z^3}$
    & $\times$ & $\times$ & $\times$\footnotemark[3] & ? \\
    \multicolumn{1}{|c|}{$\kk[x,y,z]$} & $\llrr{x^2,xy,xz,y^3,yz,z^3}$
    & $\times$ & $\times$ & \checkmark\footnotemark[3] & ? \\
    \multicolumn{1}{|c|}{$\kk[x,y,z]$} & $\llrr{x^3,xy,xz,y^2,z^2}$
    & $\times$ & \checkmark & \checkmark\footnotemark[3] & ? \\
    \multicolumn{1}{|c|}{$\kk[x,y,z]$} & $\llrr{x^2+y^2-z^2,xy,xz,yz}$
    & \checkmark & \checkmark & \checkmark\footnotemark[3] & ? \\
    \multicolumn{1}{|c|}{$\kk[x,y,z]$} & $\llrr{x^2,xy,y^2-z^2,yz}$
    & \checkmark & \checkmark & \checkmark\footnotemark[3] & ? \\
    \multicolumn{1}{|c|}{$\kk[x,y,z]$} & $\llrr{x^2,xy,y^2,z^2}$
    & \checkmark & $\times$ & \checkmark\footnotemark[3] & ? \\
    \multicolumn{1}{|c|}{$\kk[x,y,z,w]$} &
    $\llrr{x^2,xy,xz,xw,y^2,yz,yw,z^2,zw,w^3}$ & $\times$ & $\times$ &
    \checkmark\footnotemark[3] & ? \\
    \multicolumn{1}{|c|}{$\kk[x,y,z,w]$} &
    $\llrr{x^2,xy,xz,xw,y^2,yz,yw,z^2,w^2}$ & \checkmark & \checkmark
    & \checkmark\footnotemark[3] & ? \\
    \multicolumn{1}{|c|}{$\kk[x,y,z,w]$} &
    $\llrr{x^2,xy,xz,xw,y^2+z^2,y^2+w^2,yz,yw,zw}$ & \checkmark &
    \checkmark & \checkmark\footnotemark[3] & ? \\
    \multicolumn{1}{|c|}{$\kk[x,y,z,w]$} &
    $\llrr{x^2,xy-zw,xz,xw,y^2,yz,yw,z^2,w^2}$ & \checkmark &
    \checkmark & \checkmark\footnotemark[3] & ? \\ 
    \multicolumn{1}{|c|}{$\kk[x,y,z,w,v]$} & $\llrr{x,y,z,w,v}^2$ &
    \checkmark & \checkmark & \checkmark & \checkmark \\
    \hline    
  \end{tabular}
  }
  \label{table}
\end{table}

\footnotetext[3]{calculated over $\kk = \ZZ/101\ZZ$}


 \bibliography{gc-elementary}{} \bibliographystyle{amsalpha}

\providecommand{\bysame}{\leavevmode\hbox to3em{\hrulefill}\thinspace}
\providecommand{\MR}{\relax\ifhmode\unskip\space\fi MR }
\providecommand{\MRhref}[2]{%
  \href{http://www.ams.org/mathscinet-getitem?mr=#1}{#2}
}
\providecommand{\href}[2]{#2}
\begin{thebibliography}{CEVV09}

\bibitem[Bha04]{BhargavaIII--2004}
Manjul Bhargava, \emph{Higher composition laws. {III}. {T}he parametrization of
  quartic rings}, Ann. of Math. (2) \textbf{159} (2004), no.~3, 1329--1360.
  \MR{2113024}

\bibitem[Bha08]{BhargavaIV--2008}
\bysame, \emph{Higher composition laws. {IV}. {T}he parametrization of quintic
  rings}, Ann. of Math. (2) \textbf{167} (2008), no.~1, 53--94. \MR{2373152}

\bibitem[BKR01]{Bridgeland-King-Reid--2001}
Tom Bridgeland, Alastair King, and Miles Reid, \emph{The {M}c{K}ay
  correspondence as an equivalence of derived categories}, J. Amer. Math. Soc.
  \textbf{14} (2001), no.~3, 535--554. \MR{1824990}

\bibitem[BS14]{bhargava-satriano}
Manjul Bhargava and Matthew Satriano, \emph{On a notion of ``{G}alois closure''
  for extensions of rings}, J. Eur. Math. Soc. (JEMS) \textbf{16} (2014),
  no.~9, 1881--1913. \MR{3273311}

\bibitem[CEVV09]{CEVV--2009}
Dustin~A. Cartwright, Daniel Erman, Mauricio Velasco, and Bianca Viray,
  \emph{Hilbert schemes of 8 points}, Algebra Number Theory \textbf{3} (2009),
  no.~7, 763--795. \MR{2579394}

\bibitem[Che58]{Chevalley--1958}
\emph{S\'{e}minaire {C}. {C}hevalley; 2e ann\'{e}e: 1958. {A}nneaux de {C}how
  et applications}, Secr\'{e}tariat math\'{e}matique, 11 rue Pierre Curie,
  Paris, 1958. \MR{0110704}

\bibitem[EV10]{Erman--Velasco--2010}
Daniel Erman and Mauricio Velasco, \emph{A syzygetic approach to the
  smoothability of zero-dimensional schemes}, Adv. Math. \textbf{224} (2010),
  no.~3, 1143--1166. \MR{2628807}

\bibitem[Fog68]{Fogarty--1968}
John Fogarty, \emph{Algebraic families on an algebraic surface}, Amer. J. Math.
  \textbf{90} (1968), 511--521. \MR{237496}

\bibitem[Gro95]{Grothendieck--1961}
Alexander Grothendieck, \emph{Techniques de construction et th\'{e}or\`emes
  d'existence en g\'{e}om\'{e}trie alg\'{e}brique. {IV}. {L}es sch\'{e}mas de
  {H}ilbert}, S\'{e}minaire {B}ourbaki, {V}ol. 6 (1961), Soc. Math. France,
  Paris, 1995, pp.~Exp. No. 221, 249--276. \MR{1611822}

\bibitem[Hai01]{Haiman--2001}
Mark Haiman, \emph{Hilbert schemes, polygraphs and the {M}acdonald positivity
  conjecture}, J. Amer. Math. Soc. \textbf{14} (2001), no.~4, 941--1006.
  \MR{1839919}

\bibitem[Har10]{Hartshorne--2010}
Robin Hartshorne, \emph{Deformation theory}, Graduate Texts in Mathematics,
  vol. 257, Springer, New York, 2010. \MR{2583634}

\bibitem[HS20]{Ho--Satriano-2020}
Wei Ho and Matthew Satriano, \emph{Galois closures of non-commutative rings and
  an application to {H}ermitian representations}, Int. Math. Res. Not. IMRN
  (2020), no.~21, 7944--7974. \MR{4184613}

\bibitem[Hui17]{Huibregtse--2017}
Mark~E. Huibregtse, \emph{Some elementary components of the {H}ilbert scheme of
  points}, Rocky Mountain J. Math. \textbf{47} (2017), no.~4, 1169--1225.
  \MR{3689951}

\bibitem[{Hui}21]{Huibregtse--2021}
Mark~E. {Huibregtse}, \emph{{More elementary components of the Hilbert scheme
  of points}}, arXiv e-prints (2021), arXiv:2102.00494.

\bibitem[Iar72]{Iarrobino--1972}
A.~Iarrobino, \emph{Reducibility of the families of {$0$}-dimensional schemes
  on a variety}, Invent. Math. \textbf{15} (1972), 72--77. \MR{301010}

\bibitem[Iar73]{Iarrobino--1973}
\bysame, \emph{The number of generic singularities}, Rice Univ. Stud.
  \textbf{59} (1973), no.~1, 49--51. \MR{345967}

\bibitem[Iar84]{Iarrobino--1984}
Anthony Iarrobino, \emph{Compressed algebras: {A}rtin algebras having given
  socle degrees and maximal length}, Trans. Amer. Math. Soc. \textbf{285}
  (1984), no.~1, 337--378. \MR{748843}

\bibitem[Iar87]{Iarrobino--1987}
A.~Iarrobino, \emph{Hilbert scheme of points: overview of last ten years},
  Algebraic geometry, {B}owdoin, 1985 ({B}runswick, {M}aine, 1985), Proc.
  Sympos. Pure Math., vol.~46, Amer. Math. Soc., Providence, RI, 1987,
  pp.~297--320. \MR{927986}

\bibitem[IE78]{Iarrobino--Emsalem--1978}
A.~Iarrobino and J.~Emsalem, \emph{Some zero-dimensional generic singularities;
  finite algebras having small tangent space}, Compositio Math. \textbf{36}
  (1978), no.~2, 145--188. \MR{515043}

\bibitem[IK99]{Iarrobino--Kanev--1999}
Anthony Iarrobino and Vassil Kanev, \emph{Power sums, {G}orenstein algebras,
  and determinantal loci}, Lecture Notes in Mathematics, vol. 1721,
  Springer-Verlag, Berlin, 1999, Appendix C by Iarrobino and Steven L. Kleiman.
  \MR{1735271}

\bibitem[IN00]{Ito-Nakajima--2000}
Yukari Ito and Hiraku Nakajima, \emph{Mc{K}ay correspondence and {H}ilbert
  schemes in dimension three}, Topology \textbf{39} (2000), no.~6, 1155--1191.
  \MR{1783852}

\bibitem[Jel19]{Jelisiejew--2019}
Joachim Jelisiejew, \emph{Elementary components of {H}ilbert schemes of
  points}, J. Lond. Math. Soc. (2) \textbf{100} (2019), no.~1, 249--272.
  \MR{3999690}

\bibitem[Jel20]{Jelisiejew--2020}
\bysame, \emph{Pathologies on the {H}ilbert scheme of points}, Invent. Math.
  \textbf{220} (2020), no.~2, 581--610. \MR{4081138}

\bibitem[KM85]{Katz--Mazur--1985}
Nicholas~M. Katz and Barry Mazur, \emph{Arithmetic moduli of elliptic curves},
  Annals of Mathematics Studies, vol. 108, Princeton University Press,
  Princeton, NJ, 1985. \MR{772569}

\bibitem[LS67]{Lichtenbaum--Schlessinger--1967}
S.~Lichtenbaum and M.~Schlessinger, \emph{The cotangent complex of a morphism},
  Trans. Amer. Math. Soc. \textbf{128} (1967), 41--70. \MR{209339}

\bibitem[Poo08a]{Poonen--2008--iso}
Bjorn Poonen, \emph{Isomorphism types of commutative algebras of finite rank
  over an algebraically closed field}, Computational arithmetic geometry,
  Contemp. Math., vol. 463, Amer. Math. Soc., Providence, RI, 2008,
  pp.~111--120. \MR{2459993}

\bibitem[Poo08b]{Poonen--2008}
\bysame, \emph{The moduli space of commutative algebras of finite rank}, J.
  Eur. Math. Soc. (JEMS) \textbf{10} (2008), no.~3, 817--836. \MR{2421162}

\bibitem[Sag01]{sagan}
Bruce~E. Sagan, \emph{The symmetric group}, second ed., Graduate Texts in
  Mathematics, vol. 203, Springer-Verlag, New York, 2001, Representations,
  combinatorial algorithms, and symmetric functions. \MR{1824028}

\bibitem[Sha90]{Shafarevich--1990}
I.~R. Shafarevich, \emph{Deformations of commutative algebras of class {$2$}},
  Algebra i Analiz \textbf{2} (1990), no.~6, 178--196. \MR{1092534}

\bibitem[SS23]{Satriano--Staal--2023}
Matthew Satriano and Andrew~P. Staal, \emph{Small elementary components of
  {H}ilbert schemes of points}, Forum Math. Sigma \textbf{11} (2023), Paper No.
  e45, 36. \MR{4599219}

\end{thebibliography}

\end{document}